\documentclass[12pt]{amsart}

\pdfoutput=1

\usepackage[pdftex]{graphicx,color}

\usepackage{amsmath}
\usepackage{amsthm}
\usepackage{amssymb}
\usepackage{amscd}
\usepackage{amsxtra}

\newcommand{\forces}{\Vdash}

\DeclareMathOperator{\dom}{dom}
\DeclareMathOperator{\ot}{ot}
\DeclareMathOperator{\coll}{Coll}
\DeclareMathOperator{\Fil}{Fil}
\DeclareMathOperator{\ssup}{ssup}
\DeclareMathOperator{\rge}{rge}
\DeclareMathOperator{\cof}{cof}
\DeclareMathOperator{\cf}{cf}
\DeclareMathOperator{\Ult}{Ult}

\theoremstyle{plain}
\newtheorem{theorem}{Theorem}[section]
\newtheorem{lemma}[theorem]{Lemma}
\newtheorem{corollary}[theorem]{Corollary}

\theoremstyle{definition}
\newtheorem{definition}[theorem]{Definition}

\theoremstyle{remark}
\newtheorem*{remark}{Remark}
\newtheorem*{claim}{Claim}

\title{Small universal families of graphs on $\aleph_{\omega+1}$}

\author{James Cummings}

\address{Department of Mathematical Sciences, Carnegie Mellon University, Pittsburgh PA 15213-3890, USA} 

\email{jcumming@andrew.cmu.edu} 

\author{Mirna D{\v z}amonja}

\address{School of Mathematics, University of East Anglia, Norwich, NR4 7TJ,
UK}

\email{M.Dzamonja@uea.ac.uk}

\author{Charles Morgan}

\address{Department of Mathematics, University College London, Gower Street, London, WC1E 6BT, UK}

\email{charles.morgan@ucl.ac.uk}

\thanks{James Cummings  was partially supported by NSF grant DMS-1101156. Mirna D{\v z}amonja thanks EPSRC for their support through their grant
EP/I00498 and Leverhulme Trust for a Research Fellowship for the period May 2014 to May 2015.
Charles Morgan thanks EPSRC for their support through grant EP/I00498.
Cummings, D{\v z}amonja and Morgan thank the Institut Henri Poincar{\' e} for their
support through the ``Research in Paris'' program during the period 24-29 June
2013.
The authors thank Jacob Davis for his useful comments on  draft versions of this paper.}

\begin{document}

\maketitle

\begin{abstract} We prove that it is consistent that $\aleph_\omega$ is strong limit, $2^{\aleph_\omega}$ is large and the
   universality number for graphs on $\aleph_{\omega+1}$ is small. The proof uses Prikry forcing with interleaved
   collapsing.
\end{abstract}

\section{Introduction} 

   If $\mu$ is an infinite cardinal, a {\em universal graph on $\mu$} is a graph with vertex set $\mu$ which contains
   an isomorphic induced copy of every such graph. More generally, a family $F$ of graphs on $\mu$ is
   {\em jointly universal} if every graph on $\mu$ is isomorphic to an induced subgraph of some graph in $F$.
   We denote by $u_\mu$ the least size of a jointly universal family of graphs on $\mu$, and record the 
   easy remarks that $u_\mu \le 2^\mu$ and that if $u_\mu \le \mu$ then $u_\mu= 1$.
  If $\mu = \mu^{<\mu}$, then by standard results in model theory there exists a saturated
  (and hence universal) graph on $\mu$. It follows that under GCH and the hypothesis that $\mu$ is regular,
  $u_\mu = 1$. A standard idea in model theory (the construction of {\em special models}) shows that under GCH we have $u_\mu = 1$
  for singular $\mu$ as well: we fix  $\langle \mu_i : i < \cf(\mu) \rangle$ a sequence of regular cardinals which is cofinal in $\mu$,
  build a graph $\mathcal G$ which is the union of an increasing sequence of induced subgraphs
  ${\mathcal G}_i$ where ${\mathcal G}_i$ is a saturated graph on $\mu_i$, and argue by repeated applications of
  saturation that $G$ is universal.

% We note that if we force over a model of GCH to add many Cohen sets to a regular cardinal
%  $\mu$, then the saturated graph on $\mu$ from the ground model remains saturated although
%  $2^\mu$ becomes large; the point is that as the set of graphs on $\mu$ grows, so does the
%  set of embeddings into the saturated graph from the ground model.

  Questions about the value of $u_\mu$ when $\mu < \mu^{<\mu}$ have been investigated
  by several authors. We refer the reader to papers by 
  D{\v z}amonja and Shelah \cite{DzSh, DzSh2}, Kojman and Shelah \cite{KoSh},  Mekler \cite{Mekler} and Shelah \cite{Shgraphs}.

 We will consider the case when
  $\mu$ is a successor cardinal $\kappa^+$ and $2^\kappa > \kappa^+$. When $\kappa$ is regular it is known that:
\begin{enumerate}

\item  It is possible to produce models where $u_{\kappa^+}$ is arbitrarily large \cite{KoSh},
   for example by adding many Cohen subsets of $\kappa$ over a model of GCH.  
\item It is possible to produce models where $\kappa^{<\kappa} = \kappa$, $2^\kappa$ is arbitrarily
   large and $u_{\kappa^+} = \kappa^{++}$ \cite{DzSh} by iterated forcing over a model of GCH. 
\end{enumerate}
   The question whether we can have $u_{\kappa^+} = 1$ when $2^\kappa > \kappa^+$ remains mysterious
   for general values of $\kappa$, though it is known \cite{Mekler,Shgraphs} to have a positive solution for $\kappa = \omega$.  

  When $\kappa$ is singular then questions about $u_{\kappa^+}$ become harder, since we have fewer
  forcing constructions available. D{\v z}amonja and Shelah \cite{DzSh} found a line of attack on this kind of question,
  where the key idea is that  we will prepare a large cardinal $\kappa$ by means of iterated forcing
  which preserves its large cardinal character,
  and only at the end of the construction will we force to make $\kappa$ become a singular
  cardinal.  By this method D{\v z}amonja and Shelah produced models where $\kappa$ is
  singular strong limit of cofinality $\omega$, $2^\kappa$ is arbitrarily large and 
  $u_{\kappa^+} \le \kappa^{++}$.

 In \cite{DzSh} the final step in the construction is Prikry forcing, so that in the final model $\kappa$ is still
   rather large by some measures, for example it is still a cardinal fixed point.
 In this paper  we will use a forcing poset defined by Foreman and Woodin \cite{FoWo} which will make
  $\kappa$ become $\aleph_\omega$.   In some
  joint work with Magidor and Shelah \cite{5authors}, we obtain similar results where the final step is a form of Radin forcing
  which changes the cofinality of $\kappa$ to uncountable values such as $\omega_1$. 

  Our main result is this: it is consistent relative to
  a supercompact cardinal that
  $\aleph_\omega$ is strong limit, $2^{\aleph_\omega} = \aleph_{\omega+3}$,
  and $u_{\aleph_{\omega+1}} \le \aleph_{\omega+2}$. In the rest of this Introduction we give an overview
  of the proof, and conclude with a guide to the structure of the paper.

   The Foreman-Woodin poset is a variation of Prikry forcing, which 
   adds a Prikry sequence $\kappa_i$ of inaccessible cardinals cofinal in $\kappa$,
   and in addition collapses all but finitely many cardinals between successive points 
   on the Prikry sequence so that $\kappa$ becomes $\aleph_\omega$. The only parameter needed to define Prikry forcing is a normal
   measure $U_0$, but the Foreman-Woodin forcing has an additional parameter $\mathcal F$
   which is a filter on the set of functions representing elements of a certain 
   complete Boolean algebra in $\Ult(V, U_0)$.

   We will start with a ground model $V$ in which $\kappa$ is a supercompact cardinal, which  has been
   prepared so as to be indestructible under $\kappa$-directed closed forcing, and
   $2^\kappa = \kappa^{+3}$.  We will define an iterated
   forcing poset ${\mathbb Q}^*$ by  iterating
   for $\kappa^{+4}$ many steps with supports of size less than $\kappa$, forcing at each stage
   $i$ with a poset ${\mathbb Q}_i$ which is $\kappa$-directed closed and has a strong form
   of $\kappa^+$-cc. The cardinal $\kappa$ will still be supercompact in $V^{{\mathbb Q}^*}$, and
   this will enable us to choose a normal measure $U_0$ and filter $\mathcal F$, which can be
   used as parameters to define a Foreman-Woodin forcing $\mathbb P$.

   The key idea is that the poset ${\mathbb Q}_i$ will anticipate the results of
   forcing over $V^{{\mathbb Q}^*}$ with $\mathbb P$. To be more specific, at each stage $i$
   of the construction a suitable form of diamond sequence will be used to produce
   ``guesses'' $W_i$ and ${\mathcal F}_i$ at the final values of $U_0$ and $\mathcal F$,
   and there will be many stages $i$ at which these guesses are correct (in the sense
   that $W_i$ and ${\mathcal F}_i$ are the restrictions to $V^{{\mathbb P}_i}$ 
   of  $U_0$ and ${\mathcal F}$).

   At stage $i$  there is a poset ${\mathbb P}_i$ which is computed from 
   $W_i$ and ${\mathcal F}_i$ in the same way that $\mathbb P$ is computed 
   from $U_0$ and $\mathcal F$. If the guesses made at stage $i$ are correct then the final
   ${\mathbb P}$-generic object will induce a ${\mathbb P}_i$-generic object. 
   The  poset ${\mathbb Q}_i$ aims to add a $\mathbb P$-name for a graph on $\kappa^+$,
   whose interpretation  absorbs all graphs in the extension of stage $i$ by the induced ${\mathbb P}_i$-generic object.

   Our final model will be obtained by halting the construction
   at a suitable stage $i^*$ of cofinality $\kappa^{++}$, and forcing with ${\mathbb P}_{i^*}$.
   The point here (an idea which comes from \cite{DzSh}) is that we can read off
   a universal family of size $\kappa^{++}$ from a cofinal set of stages below
   $i^*$, and we are in a situation where $2^\kappa = \kappa^{+3}$.

   We conclude this section with an overview of the paper and a couple of remarks:
\begin{itemize}
\item In Section \ref{constraintsection} we discuss the filter $\mathcal F$ which is used in defining
  $\mathbb P$ and give an account of its main properties.
\item In Section \ref{pforcing} we construct the forcing ${\mathbb P}$ and prove various key facts about
  it using the properties of $\mathcal F$. 
\item In Section \ref{qsection} we construct the ``anticipation forcing'' $\mathbb Q$ and prove 
  that it has certain properties. Most notably  $\mathbb Q$ is $\kappa$-compact and has a strong form of the
  $\kappa^+$-chain condition. 
\item In Section \ref{mainsection} we describe the main iteration ${\mathbb Q}^*$ and prove a key technical fact
   by a  master condition argument.
\item In Section \ref{graphsection} we prove the main theorem.
\item In Section \ref{after} we discuss generalisations, related work and some natural open problems.
\end{itemize} 

\begin{remark}  Foreman and Woodin's paper \cite{FoWo} actually defines a supercompact Radin forcing
  with interleaved Cohen forcing, and its projection to a Radin forcing with interleaved Cohen forcing controlled
  by certain filters. Our forcing $\mathbb P$ here is a version of the projected forcing, with the Cohen forcing replaced by
  collapsing forcing and the Radin forcing simplified to the special case of Prikry forcing. $\mathbb P$ is also a close relative
  of the forcing poset used by Woodin to obtain the failure of SCH at $\aleph_\omega$ from  optimal hypotheses,
  the difference being that in Woodin's forcing poset the constraining filters are generic over the relevant
  ultrapowers. Our approach was dictated by the necessity to have the ``approximations'' ${\mathbb P}_i$
  be well-behaved forcing posets, in a context where they can neither be obtained as projections of 
  supercompact Prikry forcing with interleaved collapsing nor constructed from filters which are 
  generic over ultrapowers. Of course, all this work traces back ultimately to Magidor's original model
  for the failure of SCH at $\aleph_\omega$ \cite{SCH1}.  
\end{remark}

\section{Constraints and filters} \label{constraintsection}

   We start by assuming that $2^\kappa = \kappa^{+n}$ for some $n < \omega$ and that $\kappa$ is $2^\kappa$-supercompact.
   We will fix $U$ an ultrafilter on $P_\kappa \kappa^{+n}$ witnessing the $2^\kappa$-supercompactness of $\kappa$,
   and let $j: V \rightarrow M = \Ult(V, U)$ be the associated ultrapower map. We let $U_0$ be the projection
   of $U$ to an ultrafilter on $\kappa$ via the map $x \mapsto x \cap \kappa$.  We remind the reader of
   some standard facts. 
\begin{enumerate}
\item  $U = \{ A \subseteq P_\kappa \kappa^{+n} : j`` \kappa^{+n} \in j(A) \}$, and
       $[F]_U = j(F)(j`` \kappa^{+n})$ for every function $F$ with $\dom(F) \in U$.  
\item  $U$ concentrates on the set of $x \in P_\kappa \kappa^{+n}$ such that $x \cap \kappa$ is an
      inaccessible cardinal less than $\kappa$ and $\ot(x) = (x \cap \kappa)^{+n}$.  We will denote this
      set by $A_{\rm good}$, and for   $x \in A_{\rm good}$ we let $\kappa_x = x \cap \kappa$ and $\lambda_x = \ot(x)$. 
\item  $U_0$ is a normal measure on $\kappa$, and
       $U_0 = \{ B \subseteq \kappa : \kappa \in j(B) \}$.  We let $j_0: V \rightarrow M_0 = \Ult(V, U_0)$ be the associated ultrapower map,
       and note that  $[f]_{U_0} = j_0(f)(\kappa)$ for every function $f$ with $\dom(f) \in U_0$.
\item  There is an elementary embedding $k: M_0 \rightarrow M$ such that $k \circ j_0 = j$, which is
        given by the formula $k:[f]_{U_0}  \mapsto j(f)(\kappa)$. 
\end{enumerate}

  We now fix an integer $m$ with $n < m < \omega$, and define a family of forcing posets:
  for $\alpha$ and $\beta$ inaccessible with $\alpha < \beta$ we let
   ${\mathbb C}(\alpha, \beta) = \coll(\alpha^{+m}, < \beta)$. We note that when
   $\alpha < \beta < \gamma$ we have that  ${\mathbb C}(\alpha, \beta) \subseteq {\mathbb C}(\alpha, \gamma)$
   and the inclusion map is a complete embedding: in particular,
   if $G$ is ${\mathbb C}(\alpha, \gamma)$-generic over $V$ then $G \cap {\mathbb C}(\alpha, \beta)$ is 
   ${\mathbb C}(\alpha, \beta)$-generic over $V$.

\begin{definition}   A {\em $U$-constraint} is a function $H$ such that $\dom(H) \in U$,
   $\dom(H) \subseteq A_{\rm good}$ and $H(x) \in {\mathbb C}(\kappa_x, \kappa)$ for all
   $x \in \dom(H)$.
\end{definition} 

   It is easy to see that ${\mathbb C}^M(\kappa, j(\kappa))$ is the set of objects of the
   form $[H]_U$ for some $U$-constraint $H$. 

\begin{definition} Let $H$ and $H'$ be $U$-constraints.
\begin{enumerate}
\item   $H \le H'$ if and only if $\dom(H) \subseteq \dom(H')$ and $H(x) \le H'(x)$
    for all $x \in \dom(H)$.
\item   $H \le_U H'$ if and only if $\{ x : H(x) \le H'(x) \} \in U$,
   or equivalently $[H]_U \le [H']_U$. 
\end{enumerate}
\end{definition} 
   
\begin{remark} Since $m > n$, and ${}^{\kappa^{+n}} M \subseteq M$ by the hypothesis that $U$ witnesses the
   $\kappa^{+n}$-supercompactness of $\kappa$, 
    it is easy to see that ${\mathbb C}^M(\kappa, j(\kappa))$ is $\kappa^{+n+1}$-closed in $V$. It follows
    that
 any $\le_U$-decreasing sequence of $U$-constraints of length less than $\kappa^{+n+1}$ has a
    $\le_U$-lower bound. 
\end{remark} 

   We define the complete Boolean algebra ${\mathbb B}(\alpha, \beta)$ to be the regular open
  algebra of the forcing poset ${\mathbb C}(\alpha, \beta)$, and then let
  ${\mathbb B} = {\mathbb B}^M(\kappa, j(\kappa))$ and ${\mathbb B}_0 = {\mathbb B}^{M_0}(\kappa, j_0(\kappa))$.
  We note that for every $\alpha < \kappa$ the poset ${\mathbb C}(\alpha, \kappa)$ is $\kappa$-cc and has 
  cardinality $\kappa$, so that ${\mathbb B}(\alpha, \kappa)$ has cardinality $\kappa$:
  by elementarity we see that ${\mathbb B}_0$ has cardinality $j_0(\kappa)$ in $M_0$,
  so that in $V$ we have $\vert {\mathbb B}_0 \vert = 2^\kappa$.    

\begin{remark} Officially elements of ${\mathbb B}(\alpha, \kappa)$ are regular open subsets
    of the poset ${\mathbb C}(\alpha, \kappa)$, so that ${\mathbb B}(\alpha, \kappa)$ is not literally a
    subset of $V_\kappa$. However, since ${\mathbb C}(\alpha, \kappa)$ has the $\kappa$-chain condition,
    ${\mathbb B}(\alpha, \kappa)$ is the direct limit of the sequence of algebras
    $\langle {\mathbb B}(\alpha, \gamma) : \gamma < \kappa \rangle$, so that we may identify
    ${\mathbb B}(\alpha, \kappa)$ with a subset of $V_\kappa$. With this identification we may
    represent elements of ${\mathbb B}_0$ in the form $[h]_{U_0}$, where $h$ is a function from
    $\kappa$ to $V_\kappa$. 

    This becomes important later, when we use such functions $h$ as components of forcing
   conditions in the poset $\mathbb P$. When we move to a generic extension $W$ with the same
   $V_\kappa$ but new subsets of $\kappa$, we will need to know that $h$ can still be interpreted
   as a function which returns an element of ${\mathbb B}(\alpha, \kappa)$ on argument $\alpha$.
\end{remark}

Following Foreman and Woodin, we define a filter $\Fil(H)$ on ${\mathbb B}_0$
 from each $U$-constraint $H$.

\begin{definition}  Let $H$ be a  $U$-constraint and let $A \in U$.
  We define a function
  $b(H, A)$ as follows:
\[
   \dom(b(H, A)) = \{ \kappa_x : x \in \dom(H) \cap A \},
\]
and
\[
   b(H, A)(\alpha)  = \bigvee \{ H(x) :  \mbox{$x \in \dom(H) \cap A$ and $\kappa_x = \alpha$} \}.
\]
\end{definition} 
   In the definition of $b(H, A)(\alpha)$ we are forming the Boolean supremum of a nonempty subset of 
   ${\mathbb C}(\alpha, \kappa)$, thereby defining a nonzero element of
   ${\mathbb B}(\alpha, \kappa)$. 
  Since  $\{ \kappa_x : x \in \dom(H) \cap A \} \in U_0$, the function 
 $b(H, A)$ is defined on a $U_0$-large set and so represents a nonzero  element of the Boolean algebra ${\mathbb B}_0$ in the ultrapower $M_0$.

\begin{lemma} \label{shrinklemma}
   Let $H$ be a $U$-constraint and let $A_1, A_2 \in U$ be such that $A_2 \subseteq A_1$.
   Then   $\dom(b(H, A_2)) \subseteq \dom(b(H, A_1))$
  and $b(H, A_2)(\alpha)  \le b(H, A_1)(\alpha)$ for all
  $\alpha \in \dom(b(H, A_2))$.
\end{lemma}
\begin{proof} Straightforward.
\end{proof} 

   It follows immediately that the set $\{ [b(H, A)]_{U_0} : A \in U \}$
   forms a filter base on ${\mathbb B}_0$.

\begin{definition} 
   Let $H$ be a $U$-constraint. 
  Then  $\Fil(H)$ is the filter generated by 
$\{ [b(H, A)]_{U_0} : A \in U \}$.
\end{definition} 

\begin{lemma} \label{strengthenlemma}
  If $H_2 \le_U H_1$ then $\Fil(H_1) \subseteq \Fil(H_2)$.
\end{lemma} 
\begin{proof} Straightforward.
\end{proof} 

\begin{lemma} \label{decidelemma} For every $U$-constraint $H$ and every Boolean value $\mathfrak b$ in  
  ${\mathbb B}_0$, there is $H' \le_U H$ such that either
  $\mathfrak b \in \Fil(H')$ or $\neg  \mathfrak b \in \Fil(H')$.
\end{lemma} 

\begin{proof} We may assume that $\mathfrak b$ is non-zero.
 Let ${\mathfrak b} = [f]_{U_0}$, where $f(\alpha) \in {\mathbb B}(\alpha, \kappa)$
  and $f(\alpha)$ is non-zero  for all $\alpha \in \dom(f)$.  
  Let $A_0  = \{ x \in \dom(H) : \kappa_x \in \dom(f) \}$ and observe that 
  $A_0 \in U$. 

  For each $x$ in $A_0$, we may choose $H^*(x) \le H(x)$ such that either
  $H^*(x) \le f(\kappa_x)$ or $H^*(x) \le \neg f(\kappa_x)$. Let
  $A_1 = \{ x \in A_0 : H^*(x) \le f(\kappa_x) \}$. If $A_1 \in U$ then define
  $H' = H^* \restriction A_1$, otherwise 
  define $H'= H^* \restriction (A_0 - A_1)$. 

  If $A_1 \in U$ then consider the function $b(H', A_1)$. For every relevant $\alpha$
   we see that $b(H', A_1)(\alpha)$ is computed as a Boolean
   supremum of values which are bounded by $f(\alpha)$, so that $b(H', A_1)(\alpha) \le f(\alpha)$.
   Hence $[b(H', A_1)]_{U_0} \le [f]_{U_0}$, and accordingly
   $\mathfrak b \in \Fil(H')$. Similarly if $A_1 \notin U$ then
   $\neg \mathfrak b \in \Fil(H')$. 
\end{proof} 

\begin{lemma} \label{uflemma}  For every $U$-constraint $H$ there is $H' \le_U H$
    such that $\Fil(H')$ is an ultrafilter on ${\mathbb B}_0$.
\end{lemma}

\begin{proof}  This follows immediately from the preceding lemmas,
     the observation that $\vert {\mathbb B}_0 \vert = 2^\kappa$,
    and the fact that any $\le_U$-decreasing $2^\kappa$-sequence of $U$-constraints   has a lower bound,
\end{proof} 

\begin{lemma} \label{maxlemma}  Let $H'$ and $H''$ be $U$-constraints such that $\Fil(H')$ is an ultrafilter on ${\mathbb B}_0$ and
   $H'' \le_U H'$. Then $\Fil(H') = \Fil(H'')$.
\end{lemma} 

\begin{proof}
  Straightforward. 
\end{proof}

   It will be convenient for the arguments of Section \ref{mainsection} to formulate these ideas 
  in a slightly different language. Recall that there is an elementary embedding
  $k: M_0 \rightarrow M$ such that $k \circ j_0 = j$, given by the
  formula $k: [f]_{U_0} \mapsto j(f)(\kappa)$.

\begin{lemma} \label{naturalitylemma}
 For any $U$-constraint $H$,
\[
     \Fil(H) = \{ {\mathfrak b} \in {\mathbb B}_0 : [H]_U \le_{\mathbb B} k({\mathfrak b}) \}.
\]
\end{lemma}

\begin{proof} 
    Let $f$ be a typical function representing an element $\mathfrak b$ of ${\mathbb B}_0$,
    that is to say $\dom(f) \in U_0$ and $f(\alpha) \in {\mathbb B}(\alpha, \kappa)$ for all $\alpha$.
    Now $k({\mathfrak b}) = j(f)(\kappa)$, and $[H]_U = j(H)(j`` \kappa^{+n})$, so that
    easily $[H]_U \le j(f)(\kappa)$ if and only if $\{ x \in \dom(H) : H(x) \le f(\kappa_x) \} \in U$. 

    If ${\mathfrak b} \in \Fil(H)$ then by definition there is a set $A \in U$ such that
    $[b(H, A)]_{U_0} \le [f]_{U_0}$, that is to say 
    $B =_{\rm def} \{ \alpha : b(H, A)(\alpha) \le f(\alpha) \} \in U_0$.
    Now let $A' = A \cap \dom(H) \cap \{ x : \kappa_x \in B \}$. Clearly
    $A' \in U$; fix $x \in A'$  and observe that 
$
    H(x) \le b(H, A)(\kappa_x) \le f(\kappa_x),
$
    where the first inequality holds because $x \in \dom(H) \cap A$ 
    and the second one holds because $\kappa_x \in B$. We have shown that 
    $\{ x \in \dom(H) : H(x) \le f(\kappa_x) \} \in U$, so that $[H]_U \le_{\mathbb B} k({\mathfrak b})$.

    Conversely, if $[H]_U \le_{\mathbb B} k({\mathfrak b})$ we let $A=\{ x \in \dom(H) : H(x) \le f(\kappa_x) \}$.
    Then $\dom( b(H, A) ) = \{ \kappa_x : x \in A \}$. For every $\alpha$ in this set we have
    that
 \[
    b(H, A)(\alpha) = \bigvee \{ H(x) :  \mbox{$x \in A$ and $\kappa_x = \alpha$} \} \le f(\alpha),
\]
    where the second claim follows since (by the definition of $A$) we are forming the Boolean supremum of a set of
   values which is bounded by $f(\alpha)$.
\end{proof}

      We conclude this discussion of constraints and filters by collecting some technical facts 
    about filters of the form $\Fil(H)$ which will be useful when we define the forcing poset
    $\mathbb P$. 

\begin{definition}  A {\em $U_0$-constraint} is a partial function $h$ from $\kappa$ to $V_\kappa$ such that
   $\dom(h) \in U_0$, $\dom(h)$ is a set of inaccessible cardinals,
   and $h(\alpha) \in {\mathbb B}(\alpha, \kappa)$ for all $\alpha \in \dom(h)$.
\end{definition} 

   Clearly ${\mathbb B}_0$ is the set of objects of the form $[h]_{U_0}$ where
   $h$ is a $U_0$-constraint. 
 
\begin{definition} Let $h$ and $h'$ be $U_0$-constraints.
\begin{enumerate}
\item   $h \le h'$ if and only if $\dom(h) \subseteq \dom(h')$ and $h(\alpha) \le h'(\alpha)$
    for all $\alpha \in \dom(h)$.
\item   $h \le_{U_o} h'$ if and only if $\{ \alpha : h(\alpha) \le h'(\alpha) \} \in U_0$
   or equivalently $[h]_{U_0} \le [h']_{U_0}$. 
\end{enumerate}
\end{definition} 

\begin{lemma} \label{normalform} Let $h$ be a $U_0$-constraint and let $H$ be a $U$-constraint.
   If $[h]_{U_0} \in \Fil(H)$,  then
  there is $B \in U$ such that
   $b(H, B) \le h$.
\end{lemma}

\begin{proof} Observe that by definition there is
   $A \in U$ such that $b(H, A) \le_{U_0} h$, and define
\[
   B = \{ x \in A : b(H, A)(\kappa_x) \le h(\kappa_x) \}. 
\] 
   It is routine to check that this $B$ works.
\end{proof} 

    We now record some crucial properties of filters of the form $\Fil(H)$. In the sequel we will 
    limit attention to the special case in which $\Fil(H)$ is an ultrafilter, but only
   Lemma \ref{capturinglemma} actually requires this assumption.

\begin{lemma}[$\kappa$-completeness Lemma] \label{kappacomplete}
 Let $H$ be a $U$-constraint, let $\eta < \kappa$ and
   let $\langle h_i : i < \eta \rangle$ be a sequence of $U_0$-constraints
   such that $[h_i] \in \Fil(H)$ for all $i$. Then there exists
   a $U_0$-constraint $h$ such that 
   $[h] \in \Fil(H)$ and $h \le h_i$ for all $i$.
\end{lemma}

\begin{proof} Appealing to Lemma \ref{normalform} we choose for each
   $i < \eta$ a set $B_i \in U$ such that $b(H, B_i) \le h_i$. 
   Let $B = \bigcap_i B_i$, then $B \in U$ and it follows from
   Lemma \ref{shrinklemma} that $b(H, B) \le b(H, B_i) \le h_i$ for all $i < \eta$.
\end{proof}  

\begin{definition} 
Given a set $s \in V_\kappa$  and a $U_0$-constraint $h$, we define
\[
   h \downharpoonleft s = h \restriction \{ \alpha : s \in V_\alpha \}.
\]
\end{definition} 

\begin{lemma}[Normality lemma] \label{normalitylemma}  Let $H$ be a $U$-constraint, let $I \subseteq V_\kappa$ 
    and let  $\langle h_s : s \in I \rangle$ be an $I$-indexed
   family of $U_0$-constraints such that $[h_s]_{U_0} \in \Fil(H)$ for all $s$.
    Then there exists a $U_0$-constraint $h$ such that
    $[h]_{U_0} \in \Fil(H)$ and $h \downharpoonleft s \le h_s$ for all $s$.
\end{lemma}

\begin{proof} Choose for each $s \in I$ a set $A_s \in U$ such that
    $b(H, A_s) \le h_s$. By the normality of $U$ it follows that
if we set $A = \{ x \in \dom(H): \forall s \in I \cap V_{\kappa_x} \; x \in A_s \}$
  then $A \in U$. Let $h = b(H, A)$.

  To show this works we fix $\alpha \in \dom(h)$ and $s \in I \cap V_\alpha$.
By definition
\[
    h(\alpha) = \bigvee_{x \in A, \kappa_x = \alpha} H(x).
\]
   For every $x$ involved in this supremum we have $s \in V_{\kappa_x}$,
   so that $x \in A_s$. Hence easily
\[
    h(\alpha) = b(H, A)(\alpha) \le b(H, A_s)(\alpha) \le h_s(\alpha).
\]
\end{proof}

   With a view towards the forcing construction of Section \ref{pforcing} we define 
  the notion of {\em lower part}.

\begin{definition}
   A {\em lower part} is a finite sequence
\[
     (p_0, \kappa_1 , p_1, \ldots, \kappa_k, p_k)
\]
   such that:  
\begin{enumerate}
\item $k \ge 0$.  
\item $\kappa_i$ is an inaccessible cardinal less than $\kappa$ for $1 \le i \le k$, and
 $\kappa_1 < \kappa_2 < \ldots < \kappa_k$.
\item With the convention that $\kappa_0 = 0$ and $\kappa_{k+1} = \kappa$, 
    $p_i \in {\mathbb C}(\kappa_i, \kappa_{i+1})$ for $0 \le i \le k$. 
\end{enumerate}
\end{definition}

   Given lower parts $s$ and $t$ with
\[
      s = (p_0, \kappa_1, \ldots, \kappa_{i-1}, p_{i-1}),
\]
   we say that $t \le s$ if and only if 
\[
     t = (q_0, \kappa_1, \ldots, \kappa_{i-1}, q_{i-1})
\]
    and $q_i \le p_i$ for all $i$.

\begin{definition} A set $X$ of lower parts is 
   {\em downwards closed} if and only if for all $s \in X$ and all $t \le s$ we have $t \in X$.  
\end{definition}

    Now let us fix $H$ a $U$-constraint such that $\Fil(H)$ is an ultrafilter.   

\begin{definition}  $h$ is an {\em upper part} if and only if $h$ is
  a $U_0$-constraint such that $[h] \in \Fil(H)$. 
\end{definition} 

   The fact that $\Fil(H)$ is maximal is at the heart of the following
   crucial lemma. 

\begin{lemma}[Capturing Lemma] \label{capturinglemma} 
   Let $X$ be a downwards closed set of lower parts and let $h$ be an upper part.
   Then there exists an upper part $h^+ \le h$  such that
\begin{enumerate}
\item For all $\alpha, \beta \in \dom(h^+)$ with $\alpha < \beta$,
    $h^+(\alpha) \in {\mathbb C}(\alpha, \beta)$. 
\item For all lower parts $s$, exactly one of the two following statements holds:
\begin{enumerate}
\item For all $\alpha \in \dom(h^+)$ such that $s \in V_\alpha$, there exists
   $q \le h^+(\alpha)$ such that $s^\frown (\alpha, q) \in X$.
\item For all $\alpha \in \dom(h^+)$ such that $s \in V_\alpha$, there does not exist
   $q \le h^+(\alpha)$ such that $s^\frown (\alpha, q) \in X$.
\end{enumerate}
\item For all lower parts $s$, and all $\alpha, \beta \in \dom(h^+)$ such that $s \in V_\alpha$
    with $\alpha < \beta$, IF there is $q \le h^+(\alpha)$
    such that $s^\frown (\alpha, q) \in X$ THEN
\[
   \{ q \in {\mathbb C}(\alpha, \beta) : s^\frown (\alpha, q) \in X \}
\]
  is dense below $h^+(\alpha)$ in ${\mathbb C}(\alpha, \beta)$.  
\end{enumerate}
\end{lemma}

\begin{proof}  Strengthening $h$ if necessary, we may assume that $h = b(H, A)$
   for some $A \in U$. Fix for the moment a lower part $s$, and let
\[
     A^s \subseteq \{ x \in \dom(H) \cap A : s \in V_{\kappa_x} \}
\]
    be such that $A^s \in U$ and one of the following statements holds:

\smallskip

\noindent (Case One) For all $x \in A_s$ there is $q \le H(x)$ such that
   $s^\frown (\kappa_x, q) \in X$.

\smallskip

\noindent (Case Two) For no $x \in A_s$  is there $q \le H(x)$ such that
   $s^\frown (\kappa_x, q) \in X$.

\smallskip

   We now choose  $H^s \le H$ such that $\dom(H^s) = A^s$, and if $s$ falls in Case One
   then $s^\frown (\kappa_x, H^s(x)) \in X$ for all $x \in A^s$, and then let
   $h^s = b(H^s, A^s)$. By Lemma \ref{maxlemma},  $\Fil(H^s) = \Fil(H)$
    and so $h^s$ is a legitimate upper part. 

%CHECKED TO HERE, BAD WRITING BELOW

\medskip

\noindent  Claim One: If there exist $\alpha \in \dom(h^s)$
 and $p \le h^s(\alpha)$ such that $s^\frown (\alpha, p) \in X$, then
\[
    \{ r \in {\mathbb C}(\alpha', \kappa) : s^\frown (\alpha', r) \in X \}
\]
   is dense below $h^s(\alpha')$ in ${\mathbb C}(\alpha', \kappa)$
   for all $\alpha' \in \dom(h^s)$. 

\smallskip

\noindent Proof of Claim One: Fix some $\alpha$ and  $p \le h^s(\alpha)$ with  $s^\frown (\alpha, p) \in X$,
  and recall that $h^s(\alpha) = \bigvee_{x \in A^s, \kappa_x = \alpha} H^s(x)$.
  It follows that there is $x \in A^s$ such that $\kappa_x = \alpha$ and $p$ is comparable with $H^s(x)$,
  and we may fix $p' \le p, H^s(x)$. 
  Since $X$ is downwards closed, $s^\frown (\alpha, p') \in X$. Since $x \in A^s$ and
  $p' \le H^s(x) \le H(x)$, $s$ falls in Case One above and so
  $s^\frown (\kappa_x, H^s(x)) \in X$ for all $x \in A_s$.
 
  Let $\alpha' \in \dom(h^s)$, let  $q \in {\mathbb C}(\alpha, \kappa)$ be arbitrary with 
  $q \le h^s(\alpha')$, and observe that arguing as above there is $x \in A^s$ such that $\kappa_x = \alpha'$
  and $q$ is comparable with $H^s(x)$; if we now choose
   $r \le q, H^s(x)$ then it follows from the downwards closedness
   of $X$ and the definition of $H^s$ in Case One that $s^\frown (\kappa_x, r) \in X$.
 
\medskip

\noindent Claim Two: For all $\alpha \in \dom(h^s)$, if there is $p \le h^s(\alpha)$
    with $s^\frown (\alpha, p) \in X$ then there is an inaccessible cardinal $\beta^s(\alpha) < \kappa$
   such that $h^s(\alpha) \in {\mathbb C}(\alpha, \beta^s(\alpha))$  and
\[
    \{ r \in {\mathbb C}(\alpha, \beta) : s^\frown (\alpha, r) \in X  \}
\]
    is dense below $h^s(\alpha)$ in ${\mathbb C}(\alpha, \beta)$
   for all $\beta \ge \beta^s(\alpha)$.

\smallskip

\noindent Proof of Claim Two:  By the preceding claim,
\[
    \{ r \in {\mathbb C}(\alpha, \kappa) : s^\frown (\alpha, r) \in X \}
\]
   is dense below $h^s(\alpha)$ in ${\mathbb C}(\alpha, \kappa)$, and since
   $X$ is downwards closed this set is open. 
 Choose a maximal antichain $A$ below $h^s(\alpha)$ consisting of points in this
  set, and then appeal to the $\kappa$-chain condition to find 
  $\beta^s(\alpha) < \kappa$ such that $A \subseteq {\mathbb C}(\alpha, \beta^s(\alpha))$.

\medskip
     
 Let 
\[
    B^s = \{ \beta : \forall \alpha < \beta \; \beta^s(\alpha) < \beta \}.
\]
   Since $U_0$ is normal, $B^s \in U_0$. 
 
    To finish the proof, use Lemma \ref{normalitylemma} to find an upper part $h^-$  such that 
   $h^- \downharpoonleft s  \le  h^s$ for all $s$, and let
\[
   B = \{ \beta: \mbox{$\forall s \in V_\beta \; \beta \in B^s$ and $\forall \alpha < \beta \; h^-(\alpha) \in V_\beta$} \}.
\]
   By normality $B \in U_0$, and so we may define an upper part $h^+ = h^- \restriction B$. 
\medskip

\noindent Claim Three: $h^+$ is as required. 

\smallskip

\noindent Proof of Claim Three:  It is immediate from the definitions that $h^+ \le h$, 
    and Clause 1) from the conclusion is satisfied. 

   Towards showing Clauses 2) and 3), suppose that   $\alpha \in \dom(h^+)$,
  $s \in V_\alpha$,  $q \le h^+(\alpha)$ and $s^\frown (\alpha, q) \in X$.
  By construction $h^+(\alpha) \le h^s(\alpha)$.

  By Claim One above,  
\[
    \{ r \in {\mathbb C}(\alpha', \kappa) : s^\frown (\alpha', r) \in X \}
\]
   is dense below $h^s(\alpha')$ in ${\mathbb C}(\alpha', \kappa)$
   for all $\alpha' \in \dom(h^s)$. So if $s \in V_{\alpha'}$ and $\alpha' \in \dom(h^+)$, 
   then since  $h^+(\alpha') \le h^s(\alpha')$  this same set is dense below $h^+(\alpha')$ and so Clause 2) is satisfied.

  By Claim Two above,
\[
    \{ r \in {\mathbb C}(\alpha, \beta) : s^\frown (\alpha, r) \in X  \}
\]
    is dense below $h^s(\alpha)$ 
   for all $\beta \ge \beta^s(\alpha)$.
  If $\beta \in \dom(h^+)$ with $\alpha < \beta$ then (since $\beta \in B$)
  we have that  $h^+(\alpha) \in V_\beta$ and also $\beta \in B^s$, so that
  $\beta^s(\alpha) < \beta$ and hence 
\[
    \{ r \in {\mathbb C}(\alpha, \beta) : s^\frown (\alpha, r) \in X  \}
\]
   is dense below $h^+(\alpha)$. 
  This shows that Clause 3) is satisfied.
\end{proof} 

\begin{definition} If $X$ is a downwards closed set of lower parts and $h^+$ is an upper part satisfying the
   conclusion of the Capturing Lemma then we say that {\em $h^+$ captures $X$}.
\end{definition}

\section{The forcing $\mathbb P$ and its properties}  \label{pforcing}

\subsection{The filter}  
   In the last section we used the $2^\kappa$ supercompactness of $\kappa$
  to show that there exists a $U$-constraint $H$ such that $\Fil(H)$ is an ultrafilter.
  We then established that if $\mathcal F$ is an ultrafilter of the
   form $\Fil(H)$ then $\mathcal F$ has three properties:

\begin{enumerate}

\item[I.]  \label{kappacompletenessproperty}  ($\kappa$-completeness) 
 Let $\eta < \kappa$ and
   let $\langle h_i : i < \eta \rangle$ be a sequence of upper parts. Then there exists
   an upper part $h$ such that  $h \le h_i$ for all $i$.

\item[II.] \label{normalityproperty} (normality) 
 Let $I$ be a set of lower parts   and let  $\langle h_s : s \in I \rangle$ be an $I$-indexed
   family of upper parts.
    Then there exists an upper part $h$ such that
    $h \downharpoonleft s \le h_s$ for all $s$.

\item[III.]   \label{capturingproperty} (capturing)
   Let $X$ be a downwards closed set of lower parts and let $h$ be an upper part.
   Then there exists an upper part $h^+ \le h$ such that
   $h^+$ captures $X$.
\end{enumerate}

\begin{remark} In III above, the last part implies immediately that  if there is $q \le h^+(\alpha)$
    such that $s^\frown (\alpha, q) \in X$ then
\[
   \{ q \in {\mathbb C}(\alpha, \kappa) : s^\frown (\alpha, q) \in X \}
\]
  is dense below $h^+(\alpha)$ in ${\mathbb C}(\alpha, \kappa)$. 
\end{remark}

   For the rest of this section we will weaken our assumptions on $\kappa$,
to be precise we will assume only that:
\begin{enumerate}
\item  $\kappa$ is measurable, and $U_0$ is a normal measure on $\kappa$, with
    associated ultrapower map $j_0 : V \longrightarrow M_0 = \Ult(V, U_0)$. 
\item $2^\kappa = \kappa^{+n}$ and $n < m < \omega$.
\item  $\mathcal F$ is an ultrafilter on
    ${\mathbb B}_0 = { {\mathbb C}(\kappa, j_o(\kappa))}^{M_0}$  with properties I-III.
\end{enumerate}

\subsection{The forcing} 
   We now fix a filter $\mathcal F$ satisfying properties I-III above,
   and use $\mathcal F$ to define a forcing poset $\mathbb P$. 
   Conditions in ${\mathbb P}$ are pairs $(s, h)$ such that:
\begin{enumerate}
\item $s$ is a lower part.
 \item   $h$ is an upper part.  
\end{enumerate}  

   When $p=(s, h)$ we will refer to $s$ as the {\em stem} or {\em lower part} of $p$,
   and to $h$ as the {\em upper part} of $p$.  

 Suppose that $p=(s, h)$ and $q = (s', h')$ are conditions where
$s = (p_0, \alpha_1 , p_1, \ldots, \alpha_k, p_k)$ and
$s' = (q_0, \beta_1 , q_1, \ldots, \beta_l, p_l)$. Then $q \le p$ 
 if and only if
\begin{enumerate}
\item  $\alpha_i = \beta_i$ and $q_i \le p_i$ for $1 \le i \le k$.
\item  $\beta_i \in \dom(h)$ and $q_i \le h(\beta_i)$ for
       $k < i \le l$. 
\item  $h' \le h$.
\end{enumerate}
   $q$ is a {\em direct extension of $p$} if $q \le p$ and in addition $k =l$. 
   We write $q \le^* p$ in this case. 

The generic object for $\mathbb P$ is a sequence
\[
    f_0, \kappa_1, f_1, \kappa_2, f_2 \ldots 
\]
  where the $\kappa_i$ form an increasing and cofinal $\omega$-sequence of inaccessible cardinals less than $\kappa$ 
  (which will be generic for the Prikry forcing defined from $U_0$),   $f_i$ is ${\mathbb C}(\omega, \kappa_1)$-generic and
  $f_i$ is ${\mathbb C}(\kappa_i, \kappa_{i+1})$-generic for 
  $i > 0$. The condition $(s, h)$ where
  $s = (p_0, \alpha_1 , p_1, \ldots, \alpha_k, p_k)$ carries the information
  that
\begin{enumerate}
\item  $\kappa_i = \alpha_i$ and $p_i \in f_i$ for $1 \le i \le k$.
\item  $\kappa_i \in \dom(h)$ and $h(\kappa_i) \in f_i$ for $i > k$.
\end{enumerate}

\begin{lemma} \label{Pcclemma} The forcing poset $\mathbb P$ has the $\kappa^+$-cc.
\end{lemma}

\begin{proof} Let $(s, h)$ and $(s, h')$ be two conditions with the same stem $s$.
  Since $[h]_{U_0}, [h']_{U_0} \in {\mathcal F}$ and ${\mathcal F}$ is a filter,
   it is easy to find $h''$ such that $[h'']_{U_0} \le [h]_{U_0}, [h']_{U_0}$.
   $h'' \le_{U_0} h, h'$, and if we let
   $B = \{ \alpha : \mbox{$h''(\alpha) \le h(\alpha)$ and  $h''(\alpha) \le h'(\alpha)$} \}$
   then $h'' \restriction B \le h, h'$. The condition $(s, h'' \restriction B)$ is clearly
   a lower bound for $(s, h)$ and $(s, h')$.
\end{proof}   

   The following Lemma is straightforward.

\begin{lemma} \label{factorlemma} Let $p = (s, h)$ where $s = (p_0, \alpha_1 , p_1, \ldots, \alpha_k, p_k)$,
   and let $1 \le i \le k$. Then the forcing poset ${\mathbb P} \downarrow p$ is isomorphic to
\[
    {\mathbb D} \times ({\mathbb P}' \downarrow (t, h)),
\]
   where 
\[
    {\mathbb D} = {\mathbb C}(\omega, \alpha_1) \downarrow p_0 \times \ldots \times {\mathbb C}(\alpha_{i-1}, \alpha_i) \downarrow p_{i-1},
\]
    ${\mathbb P'}$ is defined just like $\mathbb P$ except that $\alpha_i$ plays the role of $\omega$,
and 
\[
    t = (p_i, \alpha_{i+1} , \ldots, \alpha_k, p_k).
\]
\end{lemma}

\subsection{The Prikry Lemma}

\begin{lemma}[Prikry Lemma for $\mathbb P$] Let $\Phi$ be a sentence in the forcing language and
   let $p \in {\mathbb P}$, then there is a direct extension $q \le p$ which decides $\Phi$. 
\end{lemma}

\begin{proof}  We begin the proof with a construction that is done uniformly for all
   conditions $p$.

 For each lower part $t$, if there is an upper part $h$ such that $(t, h)$ decides
    $\Phi$ then we fix such an upper part $h_t$. Appealing to Property II for $\mathcal F$,
     we find
    $h^0 \le h$ such that $h^0 \downharpoonleft t \le h^t$ for all relevant $t$. So
    for every $t$, if there exists any $h$ such that $(t, h)$ decides $\Phi$ then
    $(t, h^0)$ decides $\Phi$.

   We now define two sets of lower parts:
\[
   X^+ = \{ t : (t, h^0) \forces \Phi \},
\]
   and 
\[
   X^- = \{ t : (t, h^0) \forces \neg \Phi \}.
\]
   It is clear that both $X^+$ and $X^-$ are downwards closed. By two appeals to Property III we obtain
   $h^1 \le h^0$ such that $h^1$ captures both $X^+$ and $X^-$.

   Now let $p = (s, h)$.  As in the proof of Lemma \ref{Pcclemma}, we may find an upper
   part $h^*$ such that $h^* \le h, h^1$. Let $(t, h^{**}) \le (s, h^*)$ be a condition
   deciding $\Phi$, with $lh(t)$ chosen minimal among all such extensions
   of $(s, h^*)$.  We will show that $lh(t) = lh(s)$, establishing that
   $(t, h^{**})$ is a direct extension of $(s, h)$ and thereby proving the Lemma. 

   We will assume that $(t, h^{**}) \forces \Phi$, the proof in
   the case when it forces $\neg \Phi$ is the same. 
   Suppose for a contradiction that $lh(t) > lh(s)$, and let $t$ be the concatenation of 
   a shorter lower part $t^-$ and a pair $(\alpha, q)$. 
   Since $t$ is longer than $s$, we have that $\alpha \in \dom(h^*)$
   and $q \le h^*(\alpha) \le h^1(\alpha)$. 
   By the construction of $h^0$  we have also that $(t,  h^0) \forces \Phi$,
   so that $t \in X^+$. 

   We claim that $(t^-, h^{**}) \forces \Phi$, which will contradict the hypothesis
   that $lh(t)$ was chosen minimal and establish the Lemma. 

   Towards the claim we  observe that,  since $h^1$ captures $X^+$ and
   $q \le h^1(\alpha)$,  for every $\beta, \gamma \in \dom(h^1 \downharpoonleft t^-)$ with $\beta < \gamma$
   the set $\{ r : {t^-}^\frown (\beta, r) \in X^+ \}$ is dense
   below $h^1(\beta)$ in ${\mathbb C}(\beta, \gamma)$. We will use this to show that the set of conditions which
   force $\Phi$ is dense below $(t^-, h^{**})$ in $\mathbb P$, establishing the claim that
   $(t^-, h^{**}) \forces \Phi$.

   It will suffice to show that any extension of $(t^-, h^{**})$
   with a properly longer lower part can be extended to force $\Phi$.
   Consider such an extension of the form
   $({t'}^\frown (\gamma_0, q_0)^\frown \ldots {}^\frown (\gamma_i, q_i), h^{***})$,
   where $t' \le t^-$ and without loss of generality $i > 0$.  Since $q_0 \le h^1(\gamma_0)$ and
   $q_0 \in {\mathcal C}(\gamma_0, \gamma_1)$, by the remarks in the
   preceding paragraph there is $r \le q_0$ with
   $r \in {\mathcal C}(\gamma_0, \gamma_1)$ such that 
   $({t^-}^\frown (\gamma_0, r), h^0) \forces \Phi$.

   It is now easy to verify that by strengthening $q_0$ to $r$ we obtain a condition
      $({t'}^\frown (\gamma_0, r)^\frown \ldots {}^\frown (\gamma_i, q_i), h^{***})$
   which extends $({t^-}^\frown (\gamma_0, q), h^0)$, and so forces $\Phi$.
   This concludes the proof.
\end{proof} 

\begin{remark}  The proof of the Prikry Lemma extends without any change to the forcing poset
   ${\mathbb P}'$ defined in Lemma \ref{factorlemma}.
\end{remark} 

\subsection{Analysing names for bounded subsets of $\kappa$}

  It is clear that the forcing poset $\mathbb P$ collapses all cardinals in the
  open intervals $(\omega^{+m}, \kappa_1)$ and $(\kappa_i^{m}, \kappa_{i+1})$
  for $i > 0$. One of the main applications of the Prikry Lemma is to show that
  no other cardinals are collapsed, so that $\kappa$ becomes $\aleph_\omega$ in the
  generic extension. 

\begin{lemma} \label{bddsetlemma} 
  Let $G$ be $\mathbb P$-generic and let 
\[
    f_0, \kappa_1, f_1, \kappa_2, f_2 \ldots 
\]
   be the generic sequence added by $G$. Let $x \in V[G]$ be a bounded
   subset of $(\kappa_i^{+m})^V$ for some $i > 0$. Then
 $x \in V[f_0  \times \ldots \times f_{i-1}]$. 
\end{lemma}  

\begin{proof} Working below a suitable condition, we may use Lemma \ref{factorlemma} to view 
   $V[G]$ as a two-step extension $V[G'][g]$ where
   $g =  f_0  \times \ldots \times f_{i-1}$ and $G'$ is generic for ${\mathbb P}'$,
   a version of ${\mathbb P}$ in which $\kappa_i$ plays the role of $\omega$. 

   Let $x = i_G(\dot x)$, where $\dot x$ is a $\mathbb P$-name for a subset  
   of $\gamma$ for some $\gamma < \kappa_i^{+m}$.  We may view $\dot x$ as 
   a ${\mathbb P}'$-name for a $\mathbb D$-name for a subset of $\gamma$, where
   ${\mathbb D} = {\mathbb C}(\omega, \kappa_1) \times \ldots \times {\mathbb C}(\kappa_{i-1}, \kappa_i)$. 
   
   Since $\mathbb P'$ satisfies the Prikry Lemma, it is easy to see that
   the $\mathbb D$-name denoted by $\dot x$ lies in $V$, so that
   $x \in V[f_0 \times \ldots \times f_{i-1}]$ as required.
\end{proof} 

   Using Lemma \ref{bddsetlemma}, standard chain condition and closure arguments imply that
   only the cardinals in the intervals  $(\omega^{+m}, \kappa_1)$ and $(\kappa_i^{m}, \kappa_{i+1})$
   are collapsed by $\mathbb P$. For the purposes of some later arguments, we
   will prove a more refined version of Lemma \ref{bddsetlemma}.
   The point at stake here is that {\it a priori} it seems that a name for a bounded subset
   of $\kappa$ may depend on an arbitrarily large initial segment of the generic object,
   and this would cause major difficulties in the chain condition arguments of Section \ref{qsection}.

   Given an increasing sequence $\vec \alpha = \langle \alpha_1, \ldots, \alpha_k \rangle$ of inaccessible cardinals
   less than $\kappa$, we define
\[
    {\mathbb D}(\vec \alpha) = {\mathbb C}(\omega, \alpha_1) \times \ldots \times {\mathbb C}(\alpha_{k-1}, \alpha_k).
\]

\begin{lemma} \label{canonize} 
   Let $\mu, \eta < \kappa$ and let $\dot x$ be a $\mathbb P$-name for a subset
   of $\mu$. Let $h$ be an upper part and let
   $S$ be the set of increasing sequences $\langle \alpha_1, \ldots, \alpha_k \rangle$
   where $\alpha_i < \eta$ and $\alpha_i$ is inaccessible. 

   Then there exist an ordinal $\beta$ with  $\mu, \eta < \beta < \kappa$, names $\langle \dot y_{\vec \alpha} : \vec \alpha \in S \rangle$
  and an upper part $h' \le h$ with $\min(\dom(h')) > \beta$ such that 
  for every $\vec \alpha = (\alpha_1, \ldots, \alpha_k) \in S$:
\begin{enumerate}
\item  $\dot y_{\vec \alpha}$ is a ${\mathbb D}({\vec \alpha}^\frown \beta)$ name for a subset of $\mu$.
\item  If $t = (\emptyset , \alpha_1 ,\emptyset, \ldots, \alpha_k, \emptyset)$ then
    $(t, h') \forces \dot x = \dot y_{\vec \alpha}$. That is to say that  if $G$ is $\mathbb P$-generic
  with $(t, h') \in G$, and
\[
    f_0, \alpha_1, f_1, \alpha_2, f_2 \ldots 
\]
  is the corresponding generic sequence, then
$i_G(\dot x) = i_f(\dot y_{\vec \alpha})$, where
  $f = f_0 \times \ldots \times f_{k-1} \times (f_k \restriction \beta)$. 
\end{enumerate}    
\end{lemma}

\begin{proof} 
   As in the first step of the proof of the Prikry Lemma, we find $h^0 \le h$ such that for every
   lower part $t = (p_0, \beta_1, \ldots, \beta_k, p_k)$, if there are
   an upper part $h'$ and a ${\mathbb D}(\langle \beta_1, \ldots, \beta_k \rangle)$ -name $\dot y$ 
   such that $(t, h') \forces \dot x = \dot y$ then $(t,  h^0) \forces \dot x = \dot y$.

   For each $\vec \alpha = (\alpha_1, \ldots, \alpha_k) \in S$, each inaccessible $\delta$ with
   $\mu, \eta < \delta < \kappa$ and each canonical ${\mathbb D}(\vec \alpha^\frown \delta)$-name $\dot y$ for
   a subset of $\mu$, let $X(\vec \alpha, \delta, \dot y)$ be the set of
   lower parts $s$ such that
\[
    s =  (q_0, \alpha_1, q_1, \ldots, \alpha_k, q_k, \gamma, r)
\]
   for some $\gamma > \delta$, and $(s, h_0) \forces \dot x = \dot y$. 
   Since this is a downwards closed set of lower parts, we may find
   $h^{\vec \alpha, \delta, \dot y} \le h^0$ which captures it.
   Using Lemmas \ref{kappacomplete} and  \ref{normalitylemma} we may then find 
   an upper part $h^1$ such that $h^1 \downharpoonleft \delta \le h^{\vec \alpha, \delta, \dot y}$
   for all $\vec \alpha, \delta, \dot y$,    
   and also $\min(\dom(h^1)) > \mu, \eta$. By shrinking $\dom(h^1)$ if necessary, we will
   also arrange that $\dom(h^1)$ consists of Mahlo cardinals. 

   Fix for the moment a sequence $\vec \alpha = \langle \alpha_1, \ldots, \alpha_k \rangle \in S$.  Fix some 
   $\gamma \in \dom(h^1)$  and consider the condition
\[
   ( (\emptyset, \alpha_1, \emptyset, \ldots, \emptyset, \alpha_k, \emptyset, \gamma, h^1(\gamma)), h^1).
\]
   Working as in the proof of Lemma \ref{bddsetlemma}, we may find $r \le h^1(\gamma)$
   and $h^* \le h^1$ such that 
\[
   ( (\emptyset, \alpha_1, \emptyset, \ldots, \emptyset, \alpha_k, \emptyset, \gamma, r), h^*) \forces \dot x = \dot y
\]
   where $\dot y$ is  a canonical ${\mathbb D}(\vec \alpha^\frown \gamma)$-name for a subset of $\mu$.
   Since ${\mathbb D}(\vec \alpha^\frown \gamma)$ has the $\gamma$-cc and $\gamma$ is Mahlo, $\dot y$ is a 
   canonical ${\mathbb D}(\vec \alpha^\frown \delta)$-name for some inaccessible $\delta$ with $\mu, \eta < \delta < \gamma$. 

%NOTE TO SELF: BE CONSISTENT ABOUT PARENTHESES ROUND LOWER PARTS. PUT ``forces x equals y''
%   BITS INTO DISPLAYS!

   By construction $r \le h^1(\gamma) \le h^{\vec \alpha, \delta, \dot y}(\gamma)$.
   By the choice of $h^0$, we see that
\[
   ( (\emptyset, \alpha_1, \emptyset, \ldots, \emptyset, \alpha_k, \emptyset, \gamma, r), h^0) \forces \dot x = \dot y.
\] 
   By the choice of $h^1$, for every $\gamma_1, \gamma_2  \in \dom(h^1)$
   with $\delta < \gamma_1 < \gamma_2$  
   the set of $r^* \in {\mathbb C}(\gamma_1, \gamma_2)$ such that 
\[
   ( (\emptyset, \alpha_1, \emptyset, \ldots, \emptyset, \alpha_k, \emptyset, \gamma_1,  r^*), h^0) \forces \dot x = \dot y
\] 
   is dense
   below $h^1(\gamma_1)$. So for every $\gamma_1 \in \dom(h_1)$ 
   with $\delta < \gamma_1$
\[
   ( ( \emptyset, \alpha_1, \emptyset, \ldots, \emptyset, \alpha_k, \emptyset, \gamma_1, h^1(\gamma_1)  ), h^1) \forces \dot x = \dot y,
\] 
   which implies that
\[
   ( ( \emptyset, \alpha_1, \emptyset, \ldots, \emptyset, \alpha_k, \emptyset ), h^1 \downharpoonleft \delta) \forces \dot x = \dot y.
\] 
  To record their dependence on $\vec \alpha$,
   we write $\delta_{\vec \alpha}$ for $\delta$ and $\dot y_{\vec \alpha}$ for $\dot y$.

   Let $\beta$ be the supremum of the $\delta_{\vec \alpha}$ for $\vec \alpha \in S$, and let
   $h' = h^1 \downharpoonleft \beta$. It is now easy to see that the ordinal $\beta$, upper part
   $h'$ and family of names $\langle \dot y_{\vec \alpha} : \vec \alpha \in S \rangle$ are as
   required.

\end{proof} 

\subsection{Characterisation of genericity}

     We will need one more technical fact about ${\mathbb P}$, namely a characterisation of 
    the generic object. Similar ``geometric'' characterisations for other Prikry-type forcing posets
    appear at many places \cite{Mathias,Mitchell,James} in the literature.

\begin{lemma}[Genericity Lemma] \label{genericitylemma} Let
\[
     f_0, \kappa_1, f_1, \ldots
\]
   be such that
\begin{enumerate}
\item $f_i$ is ${\mathbb C}(\omega, \kappa_1)$-generic for $i = 0$ and
   ${\mathbb C}(\kappa_i, \kappa_{i+1})$-generic for $i > 0$.
\item For all upper parts $h$  there is
   an integer $s$ such that $\kappa_t \in \dom(h)$ and
   $h(\kappa_t) \in f_t$ for all $t \ge s$. 
\end{enumerate}

   Then this is a generic sequence for ${\mathbb P}$. 
\end{lemma}
   
\begin{proof} 

   For our later convenience we define ${\mathbb C}_0 = {\mathbb C}(\omega, \kappa_1)$ and
   ${\mathbb C}_i = {\mathbb C}(\kappa_i, \kappa_{i+1})$ for $i > 0$. 
   We make the remark that by an easy application of Easton's Lemma the filters
   $f_0, \ldots, f_n$ are mutually generic, that is
   $f_0 \times \ldots \times f_n$ is generic over $V$ for ${\mathbb C}_0 \times \ldots \times {\mathbb C}_n$.

    We now fix $E$ a dense open set in ${\mathbb P}$, with the ultimate goal of
   showing that $E$ meets the filter on ${\mathbb P}$ generated by 
\[
     f_0, \kappa_1, f_1, \ldots
\]
   To achieve this goal we need to ``canonise'' $E$ in a sense to be made precise
   later.

   By a familiar diagonal intersection argument,
    there is an upper part $h_0$ such that for every lower part $s$,
\[
       \exists h \; (s, h) \in E \iff (s, h_0 \downharpoonleft s) \in E.
\]
   Since the set $E$ is open, it is easy to see that if we let
\[
    X_0 = \{ s : (s, h_0 \downharpoonleft s) \in E \}
\]
   then $X_0$ is a downwards closed set of lower parts.  
 
    Applying Property III repeatedly we construct
   downwards closed sets $X_n$ and upper parts $h_n$ such that:
\begin{enumerate}
\item  $h_{n+1} \le h_n$.
\item  $h_{n+1}$ captures $X_n$.
\item  $X_{n+1}$ is the set of lower parts $s$ such that for some (equivalently,
  for every)  $\alpha \in \dom(h_{n+1})$ such that $s \in V_\alpha$ there is
 $q \le h_{n+1}(\alpha)$ with   $s^\frown (q, \alpha) \in X_n$.
\end{enumerate}

   We  appeal to Property I  to find an upper part $h_\infty$ such that
   $h_\infty \le h_n$ for all $n$. By the hypotheses, we  find an integer $k$ such that
   $\kappa_l  \dom(h_\infty)$ and $h_\infty(\kappa_l) \in f_l$
   for all $l \ge k$.

\begin{claim}
\[
  \{   (q_0, \ldots, q_{k-1}) : \exists j  \; (q_0, \kappa_1, \ldots, \kappa_{k-1}, q_{k-1}) \in X_j \}
\]
  is dense in  ${\mathbb C}_0 \times \ldots \times {\mathbb C}_{k-1}$. 
\end{claim}
   
\begin{proof} Let $(p_0, \ldots, p_{k-1}) \in {\mathbb C}_0 \times \ldots \times {\mathbb C}_{k-1}$,
   and consider the condition 
\[
   ( (p_0, \kappa_1, \ldots, \kappa_{k-1}, p_{k-1}, \kappa_k, h_\infty(\kappa_k)), h_\infty).
\]
   Since $E$ is dense there is an extension
\[
   ( (q_0, \kappa_1, \ldots, \kappa_{k-1}, q_{k-1}, \bar\kappa_k, q_k, \ldots, \bar\kappa_{k+j-1}, q_{k+j-1}), h) \in E
\]
   for some $j > 0$. Call the lower part of this extension $s$, and observe that by construction of $h_0$ we have
     $(s, h_0 \downharpoonleft s) \in E$ so that $s \in X_0$. Now observe that $\bar \kappa_{k+j-1} \in \dom(h_\infty)$
   and $q_{k+j-1} \le h_\infty(\bar\kappa_{k+j-1}) \le h_1(\bar\kappa_{k+j-1})$, so that
\[
    (q_0, \kappa_1, \ldots, \kappa_{k-1}, q_{k-1}, \bar\kappa_k, q_k, \ldots, \bar\kappa_{k+j-2}, q_{k+j-2}) \in X_1.
\]
   Stepping backwards in the obvious way we eventually obtain that
\[
    (q_0, \kappa_1, \ldots, \kappa_{k-1}, q_{k-1}) \in X_j.
\]
\end{proof}

   Since $f_0 \times \ldots \times f_{k-1}$ is generic, we obtain conditions $q_i \in f_i$ for 
  $i < k$ such that $t \in X_j$ where $t = (q_0, \kappa_1, \ldots, \kappa_{k-1}, q_{k-1})$. 
  Since $t \in X_j$, $\kappa_k, \kappa_{k+1} \in \dom(h_j)$  and
  $\kappa_k < \kappa_{k+1}$,
\[
   \{ p \in {\mathbb C}_k : t^\frown (\kappa_k, p) \in X_{j-1} \}
\]
   is dense below $h_j(\kappa_k)$. Also $h_j(\kappa_k) \in f_k$ because
   $h_\infty(\kappa_k) \in f_k$ and $h_\infty \le h_j$. So we may find
   $q_j^* \in f_j$ such that 
\[
    t^\frown (\kappa_j, q_j^*) \in X_{j-1}.
\]
   Repeating this argument $j$ times we construct $q_i^* \in f_i$ for 
  $k \le i < k + j$ such that
\[
   u =  t^\frown (\kappa_j, q_j^*, \ldots, \kappa_{k+j-1}, q_{k+j-1}^*) \in X_0,
\]
   that is to say that $(u, h_0 \downharpoonleft u) \in E$. 
    
   But it is now easy to verify that $(u, h_0 \downharpoonleft u)$ is in the filter
   generated by the sequence of $(f_i)$: simply observe that
\begin{enumerate}
\item $q_i \in f_i$ for $i < k$.
\item $q^*_i \in f_i$ for $k \le i < k + j$.
\item $h_0(\kappa_i) \in f_i$ for $i \ge k + j$.
\end{enumerate}
  
    This concludes the proof of the Genericity Lemma.
\end{proof}

\section{The forcing $\mathbb Q$ and its properties} \label{qsection}

   We work throughout with the same hypotheses as in Section \ref{pforcing}. 
   In particular $\mathcal F$ has properties I, II and III and $\mathbb P$ is the Prikry-type forcing defined
   from $\mathcal F$.
   Let $2^{\kappa^+} = \lambda$, and let $T$ be a tree of height $\kappa^+$ such that $T$ has at least $\lambda$ branches
   and each level of $T$ has size at most $\kappa^+$.
%
%   (OR IF WE GO THE KUREPA ROUTE AT MOST KAPPA. IN APPLICATIONS T IS THE COMPLETE BINARY TREE
%   OF HEIGHT KAPPA PLUS, DEFINED IN  A CARDINAL CORRECT INNER MODEL WHERE $2^\kappa = \kappa^+$
%   and $2^{\kappa^+} = \lambda$). 
%
   Let $\langle x_\beta : \beta < \lambda \rangle$ enumerate a sequence of distinct branches, and
   enumerate $Lev_\alpha(T)$ as $\langle t(\alpha, i) : i < \vert Lev_\alpha(T) \vert \rangle$ for
   each $\alpha < \kappa^+$.

\begin{definition} \label{harmoniousdef}
   Let $A$ be a function such that $\dom(A)$ is a bounded set of inaccessible cardinals less
  than $\kappa$, and $A(\alpha) \in {\mathbb B}(\alpha, \kappa)$ with $A(\alpha) \neq 0$ for all $\alpha \in \dom(A)$.
  Let $s = (q_0, \alpha_0, q_1, \ldots, \alpha_k, q_k)$ be a lower part, and let $\eta < \kappa$. Then
  {\em $s$ is harmonious with $A$ past $\eta$} if and only if for all $j$ such that $\alpha_j \ge \eta$
   we have $\alpha_j \in  \dom(A)$ and $q_j \le A(\alpha_j)$.
\end{definition}  

  Let $\langle \dot {\mathcal G}_\alpha: \alpha < \lambda \rangle$ enumerate all the canonical $\mathbb P$-names for 
  graphs on $\kappa^+$.  We define a forcing poset $\mathbb Q$.

  Conditions in $\mathbb Q$ are quadruples $(A, B, t, f)$ such that:

\begin{enumerate}

\item  $A$ is a function such that $\dom(A)$ is a bounded set of inaccessible cardinals less
  than $\kappa$, and $A(\alpha) \in {\mathbb B}(\alpha, \kappa)$ with $A(\alpha) \neq 0$ for all $\alpha \in \dom(A)$.

\item  $B$ is an upper part.

\item  $t$ is a triple $(\rho,  a,  b)$ where $\rho < \kappa$, $a \in [\kappa^+]^{<\kappa}$ and
    $b \in [\lambda]^{<\kappa}$.

\item  $f$ is a sequence $\langle f_{\eta, \beta} : \eta < \rho, \beta \in b \rangle$ such that
   each function $f_{\eta, \beta}$ has domain $a$.

\item  $f_{\eta, \beta} (\zeta) \in \{ x_\beta \restriction \zeta \} \times \kappa$
   for all $\eta < \rho$, $\beta \in b$ and $\zeta \in a$.

\item  \label{crucialclause} For every $\eta \in \dom(A) \cap \rho$, every lower part $s$ harmonious with $A$ past $\eta$, and every
   $\beta, \gamma \in b$ and $\zeta', \zeta \in a$ such that
   $f_{\eta, \beta}(\zeta') = f_{\eta, \gamma}(\zeta') \neq f_{\eta, \beta}(\zeta) = f_{\eta, \gamma}(\zeta)$,   
\[
    (s, B) \forces  \zeta' {\dot {\mathcal G}}_\beta \zeta \iff \zeta' {\dot {\mathcal G}}_\gamma \zeta.
\] 

\end{enumerate} 

\begin{remark}
 In the last clause, if $s$ is one of the relevant stems then all ordinals
   appearing in $s$ are less than $\ssup(\dom(A))$. 
\end{remark}

  Let $q = (A, B, t, f)$ and $q' = (A', B', t', f')$ be two conditions 
 in $\mathbb Q$. Then $q' \le q$ if and only if:
\begin{enumerate}
\item $\dom(A)$ is an initial segment of $\dom(A')$, and $A' \restriction \dom(A) = A$.
\item $B' \le B$, that is $\dom(B') \subseteq \dom(B)$ and $B'(\alpha) \le B(\alpha)$ for all $\alpha \in \dom(B')$.
\item For all $\alpha \in \dom(A') \setminus \dom(A)$, $\alpha \in \dom(B)$ and
     $A'(\alpha) \le B(\alpha)$. 
\item If we let $t = (\rho, a, b)$ and $t'=(\rho', a', b')$ then
    $\rho \le \rho'$, $a \subseteq a'$ and $b \subseteq b'$.
\item $f'_{\eta, \beta}(\zeta) = f_{\eta, \beta}(\zeta)$ for all $\eta < \rho$,
   $\beta \in b$ and $\zeta \in a$.
\end{enumerate} 

\begin{remark} The forcing poset $\mathbb Q$ is intended to add (among other things) a generic 
   function $h$ from $\kappa$ to $V_\kappa$ of the right general form to be an upper part.
   If we ultimately force with some version of $\mathbb P$ for which the generic function $h$
   is a legitimate upper part, then we will add a generic sequence $x$ which eventually obeys
   $h$ but we do not know past which point on $x$ this will begin to happen. This motivates the
   notion of ``harmonious past $\eta$'', and also explains why each $\eta$ gets its own
   set of functions $f_{\eta, \beta}$. 
 \end{remark}

%MOTIVATION:  ADD A ``GENERIC CONSTRAINT'' PLUS DATA THAT WILL ENABLE US TO BUILD A UNIV GRAPH.
%   DO NOT KNOW WHEN GENC FN STARTS TO OBEY CONSTRAINT, SO ETA AND HARM PAST.

\begin{lemma} \label{densitylemma}
   If  $G_{\mathbb Q}$ is $\mathbb Q$-generic then:
\begin{enumerate}
\item If we let $h^{G_{\mathbb Q}} = \bigcup \{ A^p : p \in {G_{\mathbb Q}} \}$ then
   $h^{G_{\mathbb Q}}$ is a function, $\dom(h^{G_{\mathbb Q}})$ is unbounded in $\kappa$, and for every
   upper part $h$ we have that  $\alpha \in \dom(h^{G_{\mathbb Q}})$ and $h^{G_{\mathbb Q}}(\alpha) \le h(\alpha)$ for all large enough 
  $\alpha \in \dom(h)$.
\item For all $\eta < \kappa$ and $\beta < \lambda$, if we  let $F^{G_{\mathbb Q}}_{\eta, \beta} = \bigcup \{ f^p_{\eta, \beta} : p \in {G_{\mathbb Q}} \}$ then
    $F^{G_{\mathbb Q}}_{\eta, \beta}$ is a function with domain $\kappa^+$.
\end{enumerate}
\end{lemma}

\begin{proof}   For the first claim, we suppose that $\nu < \kappa$, $h$ is an upper part,
   and $q$ is an arbitrary condition. Let $\mu \in \dom(B^q)$ with $\mu > \nu$, and define
   $r = (A^r, B^r, t^r, f^r)$ as follows: $A^r = A^q \cup \{ (\mu, B^q(\mu)) \}$,
   $B^r$ is some upper part such that $B^r \le B^q, h$ and $\mu < \min(\dom(B^r))$,
   $t^r= t^q$ and $f^r = f^q$. 
   
   We must verify that $r$ is a condition and $r \le q$. The only non-trivial point is to
   see that $r$ satisfies Clause \ref{crucialclause}) in the definition of  conditionhood in $\mathbb Q$. 
   Let $t$ be a lower part harmonious with $A^r$ past $\eta$.
   There are now two cases. If  $t$ is harmonious with $A^q$ past $\eta$ then $(t, B^r) \le (t, B^q)$,
   and we are done by Clause \ref{crucialclause}) for $q$. Otherwise
   $t= s^\frown \langle \mu, p \rangle$ for some $p \le B^q(\mu)$ and $s$ harmonious with $A^q$ past $\eta$, 
   $(t, B^r) \le (s, B^q)$, and again we are done by Clause \ref{crucialclause}) for $q$. 

   For the second claim, we fix $\zeta, \eta, \beta$ and then find $a \supseteq a^q$, $b \supseteq b^q$ and $\rho \ge \rho^q$ 
   such that $\eta < \rho$, $\zeta \in a$, $\beta \in b$. We then define 
   $r = (A^r, B^r, t^r, f^r)$ as follows: $A^r = A^q$, $B^r = B^q$, $t^r = (\rho, a, b)$ and
   $f^r$ is chosen to extend $f^q$ and to be such that the values
   $f^r_{\eta', \beta'}(\zeta')$ for $(\eta', \zeta', \beta') \in (\rho \times a \times b) \setminus (\rho^q \times a^q \times b^q)$
   are all distinct from each other and from any of the values $f^q_{\eta', \beta'}(\zeta')$
   for  $(\eta', \zeta', \beta') \in \rho^q \times a^q \times b^q$. This choice ensures that Clause \ref{crucialclause}) in the 
   definition of conditionhood holds, so that $r$ is a condition with $r \le q$. 
\end{proof} 

   We recall that for a regular uncountable cardinal $\nu$, a poset $\mathbb R$ is {\em $\nu$-compact} if and only if  
   the following condition holds: for every $X \subseteq {\mathbb R}$ with $\vert X \vert < \nu$,
   if every finite subset of $X$ has a lower bound then $X$ has a lower bound.

\begin{lemma} $\mathbb Q$ is $\kappa$-compact. \label{compactlemma}  
\end{lemma}

\begin{proof} Let $\mu < \kappa$, and let $\{ q_i : i < \mu \}$ be a set of conditions
   in $\mathbb Q$  such that for any finite subset $s$ of $\mu$ the set $\{ q_i : i \in s \}$ has a lower
    bound. Let $q_i = (A^i, B^i, t^i, f^i)$, and
    choose  for each finite $s \subseteq \mu$ a condition $r^s = (A^s, B^s, t^s, f^s)$ which is a
    lower bound for $\{ q_i : i \in s \}$. 

    We will define $r = (A^r, B^r, t^r, f^r)$ as follows: 
\begin{itemize}
\item  $A^r = \bigcup_{i < \mu} A^i$.
\item  $B^r$ is some upper part such that $\ssup(\dom(A^r)) < \dom(B^r)$ and
    $B^r \le B^s$ for all $s$.
\item $t^r = (\rho^r, a^r, b^r)$ where $\rho^r = \bigcup_{i < \mu} \rho^i$,  $a^r = \bigcup_{i < \mu} a^i$, $b^r = \bigcup_{i < \mu} b^i$.
\item If there is some $i$ such that $(\eta, \zeta, \beta) \in \rho^i \times a^i \times b^i$,
   then   $f^r_{\eta, \beta}(\zeta) = f^i_{\eta, \beta}(\zeta)$. As in the proof of Lemma \ref{densitylemma}, we choose 
   the values of $f^r_{\eta, \beta}(\zeta)$ for
    $(\eta, \zeta, \beta) \in \rho^r \times a^r \times b^r \setminus \bigcup_{i < \mu} (\rho^i \times a^i \times b^i)$ to be distinct
   from each other and from all values in $\{ f^r_{\eta, \beta}(\zeta) : (\eta, \zeta, \beta) \in \bigcup_{i < \mu} (\rho^i \times a^i \times b^i) \}$. 
\end{itemize}  
   We note that by our hypotheses the definition of $f^r_{\eta, \beta}(\zeta)$ yields a unique value. 
 
   As usual, the main issue is to verify that Clause \ref{crucialclause}) holds.  This is straightforward:
 if $f^r_{\eta, \beta}(\zeta) =  f^r_{\eta, \beta'}(\zeta) \neq  f^r_{\eta, \beta}(\zeta') =  f^r_{\eta, \beta'}(\zeta')$,
  then for some finite $s$ we have $f^s_{\eta, \beta}(\zeta) =  f^s_{\eta, \beta'}(\zeta) \neq  f^s_{\eta, \beta}(\zeta') =  f^s_{\eta, \beta'}(\zeta')$, 
  and so we are done because $B^r \le B^s$. 
\end{proof}

\begin{corollary} \label{compactcorollary} 
  The poset $\mathbb Q$ is $\kappa$-directed closed and also has the following property, which was
   dubbed ``parallel countable closure'' in \cite{5authors}: if $\langle q^0_i : i < \omega \rangle$ 
   and  $\langle q^1_i : i < \omega \rangle$ are decreasing sequences of conditions such that
   $q^0_i$ and $q^1_i$ are compatible for all $i$, then there is $q$ such that $q \le q^0_i, q^1_i$ for
   all $i$.
\end{corollary}

   We recall that for an regular  cardinal $\nu$, a poset $\mathbb R$ is {\em strongly $\nu^+$-cc} if and only if  
   the following condition holds: for every $\nu^+$-sequence $\langle r_i : i < \nu^+ \rangle$ of conditions in $\mathbb R$,
   there exist a club set $E \subseteq \nu^+$ and a regressive function $f$ on $E \cap \cof(\nu)$
   such that for all $i$ and $j$, if $f(i) = f(j)$ then $r_i$ is compatible with $r_j$.

\begin{lemma} $\mathbb Q$ is strongly $\kappa^+$-cc. 
\end{lemma} 

\begin{proof} Let $q^i = (A^i, B^i, t^i, f^i) \in {\mathbb Q}$  for $i < \kappa^+$, and let
   $t^i = (\rho^i, a^i, b^i)$.

    We recall that $\dom(f^i_{\eta, \beta}) = a^i$ for all $\eta < \rho^i$ and $\beta \in b^i$.
    Let $\mu^i = \ot(a^i)$.
    Let $\dot x^i_{\beta}$ be a $\mathbb P$-name for the set of pairs $(\nu, \nu')$
    such that $\zeta \dot {\mathcal G}_\beta   \zeta'$, where $\zeta$ and $\zeta'$ are respectively the
    $\nu^{\rm th}$ and $\nu'^{\rm th}$ elements of $a^i$.

    Appealing to Lemma \ref{canonize} and Property I we may assume, shrinking $B^i$ if necessary,
    that for every $\beta \in b^i$ there exist an ordinal $\gamma^i_{ \beta} < \kappa$
    and names $\dot y^i_{ \beta, \vec \alpha}$ for every increasing finite sequence $\vec \alpha$ of ordinals
    from $\ssup(\dom(A^i))$,
    such that $B^i$ ``reduces''   $\dot x^i_{\beta}$ to $\dot y^i_{ \beta, \vec \alpha}$ which is a 
    name in the product of collapses ${\mathbb D}(\vec \alpha^\frown \langle \gamma^i_{ \beta} \rangle)$
 for the edge set of a graph on
    the vertex set $\mu^i$.

    We will enumerate $\bigcup_{i < \kappa^+} b^i$ as $\langle \beta_\jmath : \jmath < \kappa^+ \rangle$.
    To make the rest of the proof more readable, we will observe the following notational
    conventions:
\begin{enumerate}
\item The letter $i$ and  its typographic variations will denote indices for conditions in $\mathbb Q$ on the sequence
$\langle q^i : i < \kappa^+ \}$.
\item  The letter $\zeta$ and its  variations will denote elements of
    $\bigcup_{i < \kappa^+} a^i$, and the letter $\beta$ and its variations will denote elements of
    $\bigcup_{i < \kappa^+} b^i$.
\item The letter $\jmath$ and its variations will denote indices for ordinals less than $\lambda$ on the 
   sequence $\langle \beta_\jmath : \jmath < \kappa^+ \rangle$.
\item
 Given a set $x \subseteq \kappa^+$ with $\vert x \vert < \kappa$, the letter $\sigma$ and its variations
   will denote indices for elements of $x$, enumerated in increasing order.
\item  Given a set $y \subseteq \bigcup_{i < \kappa^+} b^i$ with $\vert y \vert < \kappa$, the letter $\tau$
   and its variations will denote indices for elements of $\{ \jmath : \beta_\jmath \in y \}$, again enumerated
   in increasing order. Note that variations of $\tau$ denote indices (in $\kappa$) for indices (in $\kappa^+$)
   for elements of $\lambda$.
\item  The letter $\phi$ and its variations will denote indices for elements $t \in Lev_\zeta(T)$ on the sequence
     $\langle t(\zeta, \phi) : \phi < \vert Lev_\zeta(T) \rangle$.
\item  The letter $\psi$ and its variations will denote the second entries in pairs drawn from $T \times \kappa$.
\end{enumerate}

    We define functions $F_n$  with domain $\kappa^+$ for $n < 6$  as follows:
\begin{enumerate}

\item $F_0(i) = (\rho^i, \ot(a^i), \ot(\{ \jmath: \beta_{\jmath} \in b^i \}))$.  

\item $F_1(i) = a^i \cap i$.

\item $F_2(i) = \{ \jmath < i : \beta_{\jmath} \in b^i \}$.

\item $F_3(i) = A^i$.

\item $F_4(i)$ is the set of $5$-tuples
   $(\eta, \sigma, \tau, \phi, \psi)$ where
    $\eta < \rho^i$, $\sigma < \ot (a^i)$, $\tau < ot(\{ \jmath : \beta_{\jmath} \in b^i \})$, $\phi < i$, $\psi < \kappa$, 
    and if we let $\zeta$ be the $\sigma^{\rm th}$ element of $a^i$ and 
    $\beta = \beta_{\jmath}$ for $\jmath$ the $\tau^{\rm th}$ element of $\{ \jmath : \beta_{\jmath} \in b^i \}$
   then  $f^i_{\eta, \beta}(\zeta) = (t(\zeta, \phi), \psi)$.

\item $F_5(i)$ is the set of $3$-tuples $(\tau, \gamma, Y)$
   where $\tau < ot(\{ \jmath : \beta_{\jmath} \in b^i \})$,  $\gamma < \kappa$, $Y \in V_\kappa$, 
   and if we let $\beta = \beta_{\jmath}$ for $\jmath$ the $\tau^{\rm th}$ element of $\{ \jmath : \beta_{\jmath} \in b^i \}$
   then $\gamma = \gamma^i_{ \beta}$, and  $Y$ is the function specified by
   setting $Y(\vec \alpha) = {\dot y}^i_{ \beta, \vec \alpha}$ for each increasing
   finite sequence $\vec \alpha$ from $\ssup(\dom(A^i))$.

\end{enumerate}

\begin{remark} $F_4(i)$ is best viewed as a partial function on
   triples $(\eta, \sigma, \tau)$ which records a code for the value of
   $f^i_{\eta, \beta}(\zeta)$ 
 when this is ``permissible''.
 The criterion for permissibility is
   that (after decoding $\sigma$ and $\tau$ to obtain $\zeta$ and $\beta$)
  the first entry ($x_\beta \restriction \zeta$) in $f^i_{\eta, \beta}(\zeta)$ 
  is enumerated 
   before $i$ in the enumeration of level $\zeta$ of the tree $T$.
 The point is that we are aiming ultimately to define a regressive function so
   we can only record limited information.

  In a similar vein, $F_5$ is a total function which records values of
  $\gamma^i_{ \beta}$ and ${\dot y}^i_{ \beta, \vec \alpha}$.
\end{remark} 

   Now let $F(i) = (F_0(i), F_1(i), F_2(i), F_3(i), F_4(i), F_5(i))$, so that  
\[
   F(i) \in
  \kappa^3 \times [i]^{<\kappa} \times [i]^{<\kappa} \times V_\kappa \times [\kappa^3 \times i \times \kappa]^{<\kappa} \times
  [\kappa^2 \times V_\kappa]^{<\kappa}.
\]
   We fix an injective map $H$ from 
\[
   \kappa^3 \times [\kappa^+]^{<\kappa} \times [\kappa^+]^{<\kappa} \times V_\kappa \times [\kappa^3 \times \kappa^+ \times \kappa]^{<\kappa} \times
  [\kappa^2 \times V_\kappa]^{<\kappa}
\]
   to $\kappa^+$. Since $\kappa^{<\kappa} = \kappa$, we may fix a club set $E_0 \subseteq \kappa^+$ such that
   if $i \in E_0 \cap \cof(\kappa)$ then
\[
   \rge(H \restriction 
   \kappa^3 \times [i]^{<\kappa} \times [i]^{<\kappa} \times V_\kappa \times [\kappa^3 \times i \times \kappa]^{<\kappa} \times
  [\kappa^2 \times V_\kappa]^{<\kappa})
   \subseteq i,
\]
   so that in particular $H \circ F$ is regressive on $E_0 \cap cof(\kappa)$. 

   Let $E_1$ be the club subset of $\kappa^+$ consisting of those $i$ such that for all $i' < i$:  
\begin{enumerate} 
\item \label{abit}  $a^{i'} \subseteq i$.
\item \label{bbit} $\{ \jmath : \beta_{\jmath} \in b^{i'} \} \subseteq i$.
\item \label{catchupone} For all   $\zeta \in a^{i'}$ and $\beta \in b^{i'}$, $x_\beta \restriction \zeta = t(\zeta, \phi)$ for
    some $\phi < i$.
\item \label{catchuptwo} For all $\beta, \beta^* \in b^{i'}$ with $\beta \neq \beta^*$,
    $x_\beta \restriction i \neq x_{\beta^*} \restriction i$. 
\end{enumerate} 

    We claim that the  function $H \circ F$ and the club set $E_0 \cap E_1$
    serve as a witness to the strong $\kappa^+$-cc for
   $\mathbb Q$. To see this, let  $i' < i$ be points in $E_0 \cap E_1 \cap cof(\kappa)$ such that $F(i') = F(i)$.
   We will show that $q^{i'}$ and $q^{i}$ are compatible.

   We start by decoding the assertion that $F_n(i') = F_n(i)$ for $n < 4$.   Directly from the definition we see that:
\begin{enumerate}
\item $\rho^{i'} = \rho^i = \rho^*$ say.
\item $\ot(a^{i'}) = \ot(a^i) = \mu^*$ say.
\item  $\ot(\{ \jmath : \beta_{\jmath} \in b^{i'} \}) = \ot(\{ \jmath : \beta_{\jmath} \in b^i \}) = \epsilon^*$ say.
\item $a^{i'} \cap i' = a^i \cap i = r_0$ say. Since $a^{i'} \subseteq i$ by Clause \ref{abit}) in the definition
   of $E_1$, $a^{i'} \setminus {i'}$ and
    $a^{i} \setminus i$ are disjoint and $a^{i'} \cap a^i = r_0$. 
\item $\{ \jmath < {i'} : \beta_{\jmath} \in b^{i'} \} = \{\jmath < i : \beta_{\jmath} \in b^i \} = r_1$ say.
   As in the last claim,  $\{ \jmath \ge {i'} : \beta_{\jmath} \in b^{i'} \}$ and  $\{ \jmath \ge i : \beta_{\jmath} \in b^i \}$ are
   disjoint and $\{ \jmath : \beta_{\jmath} \in b^i \cap b^{i'} \} = r_1$. 
\item  $A^{i'} = A^i = A^*$ say.
\end{enumerate} 

\begin{claim} When both sides are defined, $f^{i'}_{\eta, \beta}(\zeta) =  f^{i}_{\eta, \beta}(\zeta)$.
\end{claim}

\begin{proof} We will use the fact that  $F_4(i') = F_4(i)$.
   Since both sides are defined,  $\eta < \rho^*$, $\zeta \in a^{i'} \cap a^{i}$ and 
   $\beta \in b^{i'} \cap b^{i}$. By the remarks in the preceding paragraph,
   $\zeta \in r_0$ and $\beta = \beta_{\jmath}$ for some $\jmath \in r_1$. 

   Since $r_0$ is the common initial segment of $a^{i'}$ and $a^{i}$, we have
   $\ot(a^{i'} \cap \zeta) = \ot(a^{i} \cap \zeta) = \sigma$ say. Similarly
   $\jmath$ has the same index (say $\tau$) in the increasing enumerations of
   $\{ \jmath : \beta_{\jmath} \in b^{i'} \}$ and   $\{ \jmath : \beta_{\jmath} \in b^i \}$.
     
   Now let $f^{i'}_{\eta, \beta}(\zeta) = (x_\beta \restriction \zeta, \psi')$,
   let $f^{i}_{\eta, \beta}(\zeta) = ( x_\beta \restriction \zeta, \psi)$,
   and let $x_\beta \restriction \zeta = t(\beta, \phi)$.
   Since $i \in E_1$, $i' < i$, $\beta \in b^{i'}$ and $\zeta \in a^{i'}$,
   it follows from Clause \ref{catchupone}) in the definition of $E_1$ that  $\phi < i$.
  
   By the definition of $F_4$,
   the set $F_4(i)$ contains the tuple $(\eta, \sigma, \tau, \phi, \psi)$.
    This is the
   unique tuple in $F_4(i)$ which begins with $(\eta, \sigma, \tau)$, and since $F_4(i') = F_4(i)$ this tuple also appears in $F_4(i')$.
   It follows that $\psi = \psi'$ and so 
   $f^{i'}_{\eta, \beta}(\zeta) = f^{i}_{\eta, \beta}(\zeta)$.
\end{proof}    
%
%   An argument along very similar lines gives the following assertion, which should be read
%   as saying that $q^i$ and $q^{i'}$ are isomorphic in a certain sense: MAKE IT A CLAIM
%
%  Let $\eta, \zeta, \beta, \gamma, j', j$ be such that:
%\begin{enumerate}
%
%\item $\eta < \rho^*$.
%
%\item  $\zeta \in a^{i'} \cap a^i$.
%
%\item  $\beta \in b^{i'} \setminus (b^i \cap b^{i'})$,
%   $\gamma \in  b^{i} \setminus (b^i \cap b^{i'})$,
%   
%\item $\beta = \beta_{j'}$ and $\gamma = \beta_j$.
%
%\item The index of $j'$ in the increasing enumeration of $\{ j : \beta_j \in b^{i'} \}$ is equal
%   to the index of $j$ in the increasing enumeration of $\{ j : \beta_j \in b^{i} \}$.
%\end{enumerate} 
% Then:
%\begin{enumerate}
%\item $f^{i'}_{\eta, \beta}(\zeta) = f^i_{\eta, \gamma}(\zeta)$.
%\item $\gamma^{i'}_{\eta, \beta} = \gamma^i_{\eta, \gamma}$.
%\item For each finite increasing sequence $\vec \gamma$ of ordinals
%   less than or equal to $\ssup(\dom(A^i))$,
%  ${\dot y}^{i'}_{\eta, \beta, {\vec \gamma}} = {\dot y}^{i}_{\eta, \gamma, {\vec \gamma}}$. 
%\end{enumerate} 

 We will now define $q^* = (A^*, B^*, t^*, f^*)$, which will be a lower bound for $q^i$ and $q^{i'}$.

\begin{itemize}
\item Recall that $\rho^* = \rho^{i'} = \rho^{i}$. We set $a^* = a^{i'} \cup a^i$,   $b^* = b^{i'} \cup b^{i}$,
 $t^* = (\rho^*, a^*, b^*)$. 
\item Recall that $A^* = A^{i} = A^{i'}$. 
\item Let $B^*$ be some upper part such that $B^* \le B^{i'}, B^{i}$.
\item We define $f^*_{\eta, \beta}(\zeta)$ for all $\eta < \rho^*$,  $\zeta \in a^*$ and $\beta \in b^*$.
 Naturally we set $f^*_{\eta, \beta}(\zeta) = f^{i'}_{\eta, \beta}(\zeta)$ when
  $\zeta \in a^{i'}$ and $\eta \in b^{i'}$, and similarly we set
  $f^*_{\eta, \beta}(\zeta) = f^{i}_{\eta, \beta}(\zeta)$ when $\zeta \in a^{i}$ and $\eta \in b^{i}$.
  As we argued already the sequences $f^{i'}$ and $f^{i}$ agree sufficiently for this to make sense.

  To define $f^*_{\eta, \beta}(\zeta)$ for $\eta < \rho^*$ and 
  $(\zeta, \beta) \in a^* \times b^* \setminus (a^i \times b^i \cup a^{i'} \times b^{i'})$,
  we will proceed as in the proof of Lemma \ref{compactlemma}. That is to say we will choose
  suitable values  whose second coordinates are distinct from each other, and also
  distinct from any value appearing as a second coordinate of  
  $f^*_{\eta, \beta}(\zeta)$ for $\eta < \rho^*$ and $(\zeta, \beta) \in  a^i \times b^i \cup a^{i'} \times b^{i'}$.

%  Let $Y$ be the set of $j < \kappa$ such that either
% there exist $\eta < \rho^*$,
%  $\zeta \in a^{i'}$ and $\beta \in b^{i'}$ with $f^{i'}_{\eta, \beta}(\zeta) = (x_\beta \restriction \zeta, j)$,
%  or 
% there exist $\eta < \rho^*$,
%  $\zeta \in a^i$ and $\beta \in b^i$ with $f^i_{\eta, \beta}(\zeta) = (x_\beta \restriction \zeta, j)$.
%  Note that $\vert Y \vert < \kappa$. 
%
%  Let $X = a^* \times b^* \setminus (a^{i'} \times b^{i'} \cup a^{i} \times b^{i})$, and note that
%  $\vert X \vert < \kappa$. Choose an injective map from $\rho^* \times X$ to $\kappa \setminus Y$, say
%  $\langle \delta_{\eta, \zeta, \beta} : \eta < \rho^*, (\zeta, \beta) \in X \}$.  
%  Set $f^*_{\eta, \beta}(\zeta) = \delta_{\eta, \zeta, \beta}$ for $\eta < \rho^*$ and $(\zeta, \beta) \in X$.
\end{itemize}

  It is routine to check that if $q^*$ is a condition then it is a common refinement of $q^{i'}$ and $q^{i}$,
  and also that $q^*$ satisfies all the clauses in the definition of $\mathbb Q$ except possibly for
  the final Clause \ref{crucialclause}).  With a view to verifying this clause, suppose that  $\eta \in \dom(A^*) \cap \rho^*$,
   $s$ is a lower part harmonious with $A^*$ past $\eta$,
   and 
\[
   f^*_{\eta, \beta'}(\zeta') = f^*_{\eta, \beta}(\zeta') \neq f^*_{\eta, \beta'}(\zeta) = f^*_{\eta, \beta}(\zeta).
\]
   where $\zeta' < \zeta$ and $\beta' = \beta_{\jmath'}$, $\beta = \beta_{\jmath}$ for some $\jmath' < \jmath$.  

  By the construction of $f^*$, it is immediate that all four of the pairs
  in $\{ \zeta', \zeta \} \times \{ \beta', \beta \}$
   lie in the set $a^{i'} \times b^{i'} \cup a^{i} \times b^{i}$. 

 If all four pairs above 
 lie in $a^{i'} \times b^{i'}$, then we are done by Clause \ref{crucialclause}) for $q^{i'}$ and the
  fact that $B^* \le B^{i'}$. A similar argument works if all four pairs lie in 
  $a^{i} \times b^{i}$. From this point we assume that we are not in either of these cases. 

%
% Similarly if $\zeta' \in a^{i'} \setminus r_0$
% and $\gamma \in b^{i'} \setminus \{\beta^j : j \in r_1 \}$, then all four pairs
%  lie in $a^{i'} \times b^{i'}$ and again we are done.

   Now recall that $a^* = a^{i'} \cup a^{i} = r_0 \cup (a^{i'} \setminus r_0) \cup (a^{i} \setminus r_0)$,
   where $r_0 < a^{ i'} \setminus r_0 < a^{i} \setminus r_0$. Similarly if we
   let $s = \{ \jmath : \beta_{\jmath} \in b^{i'} \cup b^{i} \}$, then
   $s = r_1 \cup ( \{ \jmath : \beta_{\jmath} \in b^{i'} \} \setminus r_1 ) 
            \cup ( \{ \jmath : \beta_{\jmath} \in b^{i} \} \setminus r_1 )$,
   where  
   $r_1 <   \{ \jmath : \beta_{\jmath} \in b^{i'} \} \setminus r_1 <  \{ \jmath : \beta_{\jmath} \in b^{i} \} \setminus r_1$.
  
\begin{figure}[h]

\begin{center}

\resizebox{4in}{!}{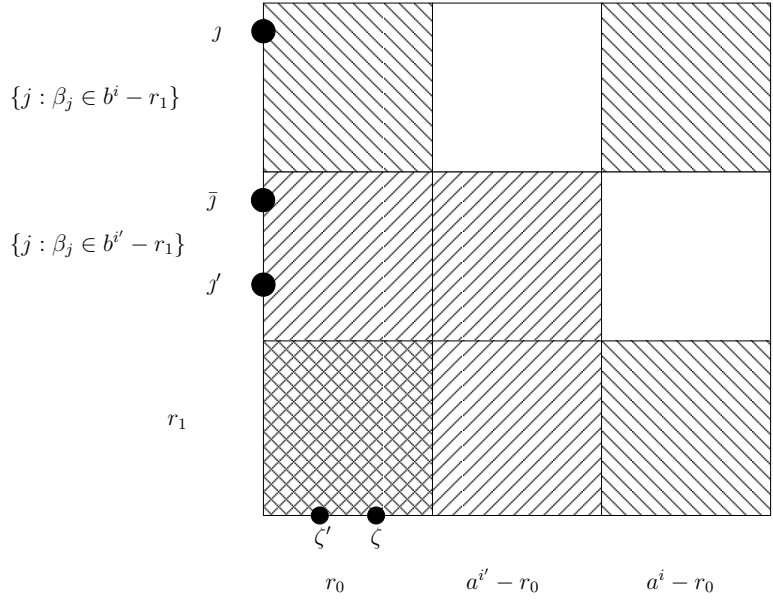} 

\caption{Subcase 2b: $\bar\jmath$ is the ``clone'' of $\jmath$.}

\label{myfigure}

\end{center}

\end{figure}

   An easy case analysis
\footnote{In figure \ref{myfigure}, all pairs $(\zeta^*, \jmath^*)$ with $(\zeta^*, \beta_{\jmath^*}) \in a^{i'} \times b^{i'}$
    lie in the region shaded with forward-sloping diagonal lines, and 
   all pairs $(\zeta^*, \jmath^*)$ with $(\zeta^*, \beta_{\jmath^*}) \in a^{i} \times b^{i}$
     lie in the region shaded with backward-sloping diagonal lines.
    Points in $\{ \zeta' \zeta \} \times \{ \jmath', \jmath \}$ must
   all lie in the shaded region, and must not all lie in subregions shaded in a single direction.} shows that 
    there are only two possibilities:

\smallskip
  
\noindent {\bf Case 1:} $\jmath'$ and $\jmath$ are both in $r_1$, $\zeta' \in a^{i'} \setminus r_0$,
     $\zeta \in a^i \setminus r_0$.

\smallskip

\noindent {\bf Case 2:} $\zeta'$ and $\zeta$ are both in $r_0$, $\jmath' \in b^{i'} \setminus r_1$,
     $\jmath \in b^i \setminus r_1$.

\smallskip

   We will first show that Case 1 is not possible.
   To dismiss Case 1, assume that we are in this case and recall that 
\[
   f^*_{\eta, \beta'}(\zeta') = f^*_{\eta, \beta}(\zeta') \neq f^*_{\eta, \beta'}(\zeta) = f^*_{\eta, \beta}(\zeta),
\]
  from which it follows that $x_{\beta'} \restriction \zeta = x_\beta \restriction \zeta$. 
  Since $i \in E_1$, $i' < i$ and  $\beta', \beta \in b^{i'}$,
  it follows from clause \ref{catchuptwo} in the definition of $E_1$  that
  $x_{\beta'} \restriction i \neq x_\beta \restriction i$. By the case assumption 
  we have $\zeta \in a^i \setminus r_0 = a^i \setminus (a^i \cap i)$, so that
  in particular $\zeta \ge i$. This is a contradiction, so Case 1 does not occur.

   We now assume that we are in Case 2. In order to use the information coded in the
   equality of $F(i')$ and $F(i)$, we make some definitions:
\begin{itemize}
\item  $\sigma'$ is the index of $\zeta'$ in the increasing enumeration of $a^{i'} \cap a^i$.
\item  $\sigma$ is the index of $\zeta$ in the increasing enumeration of $a^{i'} \cap a^i$.
\item  $\tau'$ is the index of $\jmath'$ in the increasing enumeration of
    $\{ \jmath : \beta_{\jmath} \in a^{i'} \}$.
\item  $\tau$ is the index of $\jmath$ in the increasing enumeration of
    $\{ \jmath : \beta_{\jmath} \in a^i \}$.
\end{itemize}

  By definition, $F_5(i')$ (which is equal to $F_5(i)$) contains the tuples
 $(\tau', \gamma^{i'}_{ \beta'}, Y')$ and
 $(\tau,  \gamma^{i}_{ \beta},    Y)$, where
 $Y'(\vec \alpha) = {\dot y}^{i'}_{ \beta', \vec \alpha}$ 
 and $Y(\vec \alpha) = {\dot y}^i_{ \beta, \vec \alpha}$
  for $\alpha$ any increasing finite sequence from $\ssup(\dom(A))$.

  Recall that $s$ is a lower part harmonious with $A^*$ past $\eta$,
  and $\eta \in \dom(A^*) \cap \rho^*$. Let 
\[
   s =  (p_0, \alpha_1 , p_1, \ldots, \alpha_k, p_k),
\]
  and let $\vec \alpha = \langle \alpha_1, \ldots, \alpha_k \rangle$.
  By the choice of $B^{i'}$ and the definitions of the names ${\dot x}^{i'}_{\beta'}$ and
  ${\dot y}^{i'}_{ \beta', \vec \alpha}$, $(s, B^{i'})$ reduces the truth
  value of $\zeta' {\dot {\mathcal G}}_{\beta'} \zeta$ 
  (a Boolean value for $\mathbb P$) to the truth value of
  $\sigma'  {\dot y}^{i'}_{ \beta', \vec \alpha} \sigma$ 
  (a Boolean value for the product of collapses
     ${\mathbb D}(\vec \alpha^\frown  \gamma^{i'}_{ \beta'})$).
  Similarly $(s, B^{i'})$ reduces the truth value of
  $\zeta'  {\dot {\mathcal G}}_\beta \zeta$ to the truth value of
  $\sigma' {\dot y}^{i}_{ \beta, \vec \alpha} \sigma$. 

\smallskip

\noindent {\bf Subcase 2a:} $\tau' = \tau$. 

\smallskip

   In this subcase $\gamma^{i'}_{ \beta'} = \gamma^{i}_{ \beta} = \gamma^*$ say, and
   $Y' = Y$, 
   so that in particular ${\dot y}^{i'}_{ \beta', \vec \alpha} = {\dot y}^i_{ \beta, \vec \alpha}$.
   It is then immediate from the preceding discussion that, since $(s, B^*)$ is a common refinement of  $(s, B^{i'})$ and  $(s, B^i)$,
    $(s, B^*) \forces \zeta' {\dot {\mathcal G}}_{\beta'}  \zeta  
    \iff \zeta' {\dot {\mathcal G}}_{\beta}  \zeta$.  

\smallskip

\noindent {\bf Subcase 2b:} $\tau' \neq \tau$.    

\smallskip

    In this subcase we will consider a ``cloned'' version $\bar \beta$ of $\beta$ lying in
    $b^{i'}$, which we define by setting $\bar \beta = \beta_{\bar \jmath}$ for $\bar \jmath$ the
    element with index $\tau$ in the increasing enumeration of 
    $\{  \jmath : \beta_{\jmath} \in b^{i'} \}$. 
    The argument from subcase 2a shows that 
    $(s, B^*) \forces
 \zeta' {\dot {\mathcal G}}_{\bar\beta}  \zeta  
    \iff \zeta' {\dot {\mathcal G}}_{\beta}  \zeta$.

%    CHECK INDICES CAREFULLY HERE

    Since $\beta' \in b^{i'}$ and $\zeta' \in a^{i'}$,
    and also $i' < i$ and $i \in E_1$, it follows
    from Clause \ref{catchupone} in the definition of $E_1$ 
    that  $x_{\beta'} \restriction \zeta' = t(\zeta', \phi)$
    for some $\phi < i$. Now since   $f^{i'}_{\eta, \beta'}(\zeta') = f^i_{\eta, \beta}(\zeta')$,
    $x_{\beta'} \restriction \zeta' = x_\beta \restriction \zeta = t(\zeta', \phi)$, so that
    the set $F_4(i)$ contains some tuple $(\eta, \sigma', \tau, \phi, \psi)$
    coding the statement ``$f^i_{\eta, \beta}(\zeta') = (t(\zeta', \phi), \psi)$''.
    This tuple is also in $F_4(i')$, and decoding its meaning we find that
    $f^{i'}_{\eta, \bar\beta}(\zeta') = (t(\zeta', \phi), \psi) = f^i_{\eta, \beta}(\zeta')$.
    Since $\zeta \in a^{i'}$ also, a similar argument shows that
    $f^{i'}_{\eta, \bar\beta}(\zeta)  = f^i_{\eta, \beta}(\zeta')$. 

    So now we have $f^{i'}_{\eta, \beta'}(\zeta') = f^{i'}_{\eta, \bar\beta}(\zeta')$
    and $f^{i'}_{\eta, \beta'}(\zeta') = f^{i'}_{\eta, \bar\beta}(\zeta')$, so
    that (by Clause \ref{crucialclause})  for the condition $q^{i'}$) 
  $(s, B^{i'})
\forces
 \zeta' {\dot {\mathcal G}}_{\beta'}  \zeta  
    \iff \zeta' {\dot {\mathcal G}}_{\bar\beta}  \zeta$.

   So  $(s, B^*)
\forces
 \zeta' {\dot {\mathcal G}}_{\beta'}  \zeta  
    \iff \zeta' {\dot {\mathcal G}}_{\beta}  \zeta$.  
 and we are done.

\end{proof} 

\section{The main construction} \label{mainsection}

%    DONT FORGET THE STARTING STEPS WHEN WE GET TO STRETCHING AND MC ARGT. 

    We will start with a model $V_0$ in which GCH holds and $\kappa$ is supercompact. In this model
    we define in the standard way \cite{Laver} a ``Laver preparation'' forcing $\mathbb L$, and let
    $V_1 = V_0[G_0]$ where $G_0$ is $\mathbb L$-generic over $V_0$.
    Let $\mathbb A$ be the poset $Add(\kappa^+, \kappa^{+3})^{V_1}$, let
    $G_1$ be $\mathbb A$-generic over $V_1$, and let $V = V_1[G_1]$.

    Let  $T$ be the complete binary tree of height $\kappa^+$
    as defined in $V$. Clearly $T$ has $\kappa^{+3}$ branches and every level of $T$ has size
    at most $\kappa^+$. For use later we fix an enumeration 
    $\langle x_\beta : \beta < \kappa^{+3} \rangle$ of a set of distinct branches,
    and enumerations $\langle t(\alpha, i) : i < \vert Lev_\alpha(T) \vert \rangle$
    of the levels $Lev_\alpha(T)$ of $T$.
 
    Since $2^{\kappa^{+3}} = \kappa^{+4}$ in $V$, a theorem of Shelah \cite{Shelahdiamonds} implies that
    in $V$ we have $\diamondsuit_{\kappa^{+4}}(\cof(\kappa^{++}))$.  

    Working in $V$, we will define a forcing iteration with $<\kappa$-supports of length $\kappa^{+4}$.
    Each iterand ${\mathbb Q}_i$ will either be trivial forcing or will be $\kappa$-closed, parallel countably closed
    (in the sense of Corollary \ref{compactcorollary}) and strongly
    $\kappa^+$-cc. By a suitable adaptation of arguments of Shelah \cite{ShFA},
    this is sufficient to show that the whole iteration will be $\kappa$-directed closed
    and strongly $\kappa^+$-cc. We refer the reader to \cite{5authors} for a detailed account
    of the chain condition proof, noting (for the experts) that the property ``parallel countably closed''
    follows from the property ``countably closed plus well-met'' used in \cite{ShFA} and is 
    sufficient to make the proof from that paper work.

 The cardinality of the final iteration ${\mathbb Q}^*$ will
    be $\kappa^{+4}$.  We will have $2^\kappa = \kappa^{+4}$ in $V^{{\mathbb Q}^*}$, 
    while $2^\kappa = \kappa^{+3}$ in the intermediate models of the the iteration.  
    We note that by the closure of ${\mathbb Q}^*$,
    the terms ``$V_\kappa$'' and ``$P_\kappa \mu$'' have the same meanings in
    $V$, $V^{{\mathbb Q}^*}$, and every intermediate model. 

    As we build the iteration ${\mathbb Q}^*$, we will also (using the diamond from $V$) construct a sequence 
    of names ${\dot {\mathbb S}}_i$ such that
\begin{itemize}
\item ${\dot {\mathbb S}}_i$ is a ${\mathbb Q}^* \restriction i$-name for every $i < \kappa^{+4}$.
\item ${\dot {\mathbb S}}_i$ names a pair $(W_i, F_i)$ where $W_i \subseteq P(\kappa)$, $F_i$ is a family
    of partial functions from $\kappa$ to $V_\kappa$, and $\dom(H) \in W_i$ for all $H \in F_i$.
\item  If $G^*$ is ${\mathbb Q}^*$-generic, and $(W, F) \in V[G^*]$ with $W \subseteq P(\kappa)$ 
   and $F$ a family of functions from sets in $W$ to $V_\kappa$, then
 \[
    \{ i \in \kappa^{+4} \cap \cof(\kappa^{++}) :
    \mbox{$W \cap V[G^* \restriction i] = W_i$ and $F \cap V[G^* \restriction i] = F_i$} \} 
\]
   is stationary in $V[G^*]$.
\end{itemize}

    This is possible because:
\begin{itemize} 
\item Pairs $(W, F)$ as above in the extension by ${\mathbb Q}^*$ may be coded as subsets of
   $\kappa^{+4}$, and names for them may be coded as subsets of ${\mathbb Q}^* \times \kappa^{+4}$. 
\item If we enumerate the conditions in ${\mathbb Q}^*$ as $\langle q_j : j < \kappa^{+4} \rangle$, then 
    ${\mathbb Q}^* \restriction i = \{ q_j : j < i \}$ for almost all $i \in \kappa^{+4} \cap cof(\kappa^{++})$.  
\item ${\mathbb Q}^*$ preserves stationary subsets of $\kappa^{+4}$, by virtue of being $\kappa^+$-cc.
\end{itemize} 
    We refer the reader to \cite{5authors} for a more detailed discussion of this kind of construction.

  We observe that by $\kappa^+$-cc, if $i < \kappa^{+4}$ with  $\cf(i) > \kappa$ then every subset
  of $V_\kappa$ in the extension by ${\mathbb Q^*} \restriction i$ is in the  
  extension by ${\mathbb Q^*} \restriction j$ for some $j < i$.
  We observe also that each of the properties I-III can be formulated as $\forall \exists$
  assertions about the power set of $V_\kappa$.

  The considerations in the last paragraph imply a crucial reflection statement for Properties I-III:
  If $G^*$ is ${\mathbb Q}^*$-generic,  in $V[G^*]$ we have a normal
  measure $U_0$ and filter $\mathcal F$ with properties I-III, and we set 
  $F = \{ h : [h]_U \in {\mathcal F} \}$, then for almost all $i$ with
  $\cf(i) > \kappa$ we have that:
\begin{itemize} 
\item $U_0 \cap V[G^* \restriction i]$ and $F \cap V[G^* \restriction i]$ are elements 
   of $V[G^* \restriction i]$.
\item In $V[G^* \restriction i]$, $U_0 \cap V[G^* \restriction i]$ is a normal measure and
   the functions in $F \cap V[G^* \restriction i]$ represent a filter with
   properties I-III. 
\end{itemize}

After these preliminaries we can specify the iterands ${\mathbb Q}_i$ of the iteration ${\mathbb Q}^*$.
 We   assume that $G^* \restriction i$ is ${\mathbb Q}^* \restriction i$-generic and that
 $(W_i, F_i)$ is the realisation of ${\dot {\mathbb S}}_i$, and work in $V[G^* \restriction i]$.
 We will set
 ${\mathbb Q}_i$ to be trivial forcing unless we have the conditions:
\begin{itemize}
\item  $\cf(i) = \kappa^{++}$.
\item  $W_i$ is a normal measure on $\kappa$.
\item  $\{ [h]_{W_i}: h \in F_i \}$ is an ultrafilter satisfying properties I-III.
\end{itemize}
  In this case we will let ${\mathbb Q}_i$  be the forcing $\mathbb Q$ from Section
  \ref{qsection}, defined in $V[G^* \restriction i]$ from the parameters $W_i$, $F_i$, $\langle x_\beta : \beta < \kappa^{+3} \rangle$,
  and a suitable enumeration  $\langle {\dot {\mathcal G}}^i_\beta : \beta < \kappa^{+3} \rangle$ of canonical names for graphs.  

   We recall from the Introduction we will ultimately force over
   $V[G^*]$ with a poset $\mathbb P$ of the type discussed in Section
   \ref{pforcing}. The forcing $\mathbb P$ will be defined from some normal measure $U_0$ and ultrafilter $\mathcal F$,
   and the point of the diamond machinery in the definition of ${\mathbb Q}_i$ is to
   anticipate the poset $\mathbb P$ (and in particular $\mathbb P$-names for
   graphs on $\kappa^+$). To be a bit more precise, we will actually anticipate
   $U_0$ and $F$ where  $F = \{ h : [h]_{U_0} \in {\mathcal F} \}$, or to put it another way
   $F$ is the set of upper parts for the poset $\mathbb P$.  

   Recall further from Section \ref{qsection} that in the case when
   ${\mathbb Q}_i$ is not trivial forcing,  part of the generic object for
   ${\mathbb Q}_i$ will be a partial function $h_i$ from $\kappa$ to
   $V_\kappa$ such that
\begin{itemize}
\item $\dom(h_i)$ is an unbounded set of inaccessible cardinals 
\item $\dom(h_i)$ is eventually contained in each measure one set for the measure $W_i$.
\item For all $h \in F_i$, $h_i(\alpha) \le h(\alpha)$ for all large enough $\alpha \in \dom(h)$.
\end{itemize}  

   The following Lemma will be used in Section \ref{graphsection} to show that often enough ${\mathbb Q}_i$ does 
   its job, by adding a $\mathbb P$-name for a graph which will absorb all
   graphs whose names lie in $V[G^* \restriction i]$.  

\begin{lemma}  Let $G^*$ be ${\mathbb Q}^*$-generic. Then in $V[G^*]$ there is an ultrafilter $U$ 
  on $P_\kappa \kappa^{+4}$ such that if $U_0$ is the projection of $U$ to $\kappa$,
   and  we perform the construction of Section \ref{constraintsection}
   to produce an ultrafilter ${\mathcal F} = Fil(H)$ for some $U$-constraint $H$, then there are stationarily many
   $i < \kappa^{+4}$ such that:
\begin{enumerate} 
\item  $U_0 \cap V[G^* \restriction i]=W_i$.
\item  $F \cap V[G^* \restriction i]=F_i$, where $F = \{ h : [h]_{U_0} \in {\mathcal F} \}$.
\item  $h_i \in F$.
\end{enumerate} 
\end{lemma} 

   Before starting the proof, we emphasise that the diamond property ensures that there many $i$ where the
   first two clauses are satisfied. What takes work is arranging that the the third clause 
   is also satisfied. 

\begin{proof} 

   We will construct $U$ as in the standard proof of Laver's indestructibility result \cite{Laver}, with the
   proviso that we will be very careful about the construction of the master condition.

   We begin by falling back to the initial model $V_0$, where we will choose an embedding
   $j : V_0 \rightarrow M$ with critical point $\kappa$ witnessing that $\kappa$ is $\mu$-supercompact for some very large $\mu$,
   and with the additional properties that the forcing poset ${\mathbb A} * {\mathbb Q}^*$ is the iterand at stage
   $\kappa$ in the iteration $j({\mathbb L})$, and that the least point greater than $\kappa$ in the
   support of the iteration $j({\mathbb L})$ is greater than $\mu$.  
   
   Recall that $V = V_0[G_0][G_1]$ where $G_0$ is ${\mathbb L}$-generic and $G_1$ is $\mathbb A$-generic. 
   By standard arguments, for any choice of a generic object $H_{\rm tail}$ for $j({\mathbb L})/G_0 * G_1 * G^*$
   over the model $V[G^*]$, we have $j`` G_0 \subseteq G_0 * G_1 * G^* * H_{\rm tail}$. 
   We may therefore  lift $j$ to obtain a generic embedding
   $j: V_0[G_0] \rightarrow M[G_0 * G_1 * G^* * H_{\rm tail}]$.  In order to lift further, we
   will need to construct master conditions.

   We will now work in $V[G^*]$ and perform a recursive construction of length $\kappa^{+4}$, choosing a decreasing sequence of conditions
   $(r_i, a_i, q_i)$
   with $(r_i, a_i, q_i)  \in j({\mathbb L})/(G_0 * G_1 * G^*)  * j({\mathbb A}) * j({\mathbb Q}^* \restriction i)$. We will arrange that
\[
    r_i \forces a_i \le j`` G_1,
\]
   so that forcing below $(r_i, a_i)$ we obtain $H_{\rm tail} * H_1$ such that $j$ can be lifted
   to $j: V_0[G_0][G_1] \rightarrow M[G_0 * G_1 * G^* * H_{\rm tail} * H_1]$. Keeping this in mind, we will
   also arrange that  
  
\[
    (r_i, a_i) \forces q_i \le j`` (G^* \restriction i).
\]

   Using the hypothesis that $j$ witnesses
   $\mu$-supercompactness and the remark that $G_1 \in M[j(G_0)]$, we may argue that
  for any choice of $H_{\rm tail}$ we have  $j`` G_1 \in M[G_0 * G_1 * G^* * H_{\rm tail}]$. 
   Since $j(\mathbb A)$ is $j(\kappa^+)$-directed closed we may find a ``strong master condition'' 
   $a \in j(\mathbb A)$ with $a \le j`` G_1$. We will therefore choose 
   $r_0$ to be the trivial condition in $j({\mathbb L})/(G_0 * G_1 * G^*)$, $a_0$ to be (a name
   for) a condition $a \in j(\mathbb A)$ with $a \le j`` G_1$, and $q_0$ as (a name for) the empty sequence.

   The limit stages are straightforward, since the choice of $\mu$ and  $j$ gives enough closure to
   take lower bounds.  
   If ${\mathbb Q}_i$ is trivial it is easy to define suitable $r_{i+1}$, $a_{i+1}$ and $q_{i + 1}$. so we assume that
   ${\mathbb Q}_i$ is non-trivial.
  
    Forcing below $(r_i, a_i, q_i)$ we can obtain a generic object $H_{\rm tail} * H_1 * H^*_i$ such that
    there is a lifted embedding $j:V[G^* \restriction i] \rightarrow M[j(G_0 * G_1) * H^*_i]$,
    where $j(G_0 * G_1) = G_0 * G_1 * G^* * H_{\rm tail} *  H_1$. 
    Let $g_i$ be the ${\mathbb Q}_i$-generic filter added at stage $i$ by $G^*$, and let    
    $h_i$ be the partial function from $\kappa$ to $V_\kappa$ added by $g_i$. 
 
   To take the next step, we ask whether it is possible that the set
    $\{ j(h)(\kappa) : h \in F_i \}$ has a non-zero lower bound: more formally, we ask whether
    there is a condition extending $(r_i, a_i, q_i)$ which forces this set to have a non-zero lower bound,
    and define $(r', a', q')$ to be such a condition if it exists and to be  
    $(r_i, a_i, q_i)$ otherwise. In the case that $(r', a', q')$ forces that 
    $\{ j(h)(\kappa) : h \in F_i \}$ has a non-zero lower bound, we let $\mathfrak b$ name the
    Boolean greatest lower bound for this set.  In either case we force below
    $(r', a', q')$, lift $j$ and work in $M[j(G_0 * G_1) * H^*_i]$ to define a condition
    in $j({\mathbb Q}_i)$. 

    Let  $\langle f^i_{\eta, \beta} : \eta < \kappa, \beta < \kappa^{+3} \rangle$ 
   be the family of functions added by $g_i$.
   We  define $Q = (A^Q, B^Q, t^Q, f^Q)$
% hopefully  in $j({\mathbb Q}_i)$
   as follows:
\begin{itemize}

\item  $t^Q = \kappa \times j`` \kappa^+ \times j``\kappa^{+3}$.

\item  For all $\eta < \kappa$, $\alpha < \kappa^+$ and $\beta < \kappa^{+3}$,
   $f^Q_{\eta, j(\beta)}(j(\alpha)) = j(f^i_{\eta, \beta}(\alpha))$. 

\item   If the Boolean value $\mathfrak b$ is not defined,
   then:
\begin{enumerate}
\item  $A^Q = h_i$.
\item  $B^Q$ is some upper part such that $\kappa \cap \dom(B^Q) = 0$ and
    $B^Q \le j(B)$ for all upper parts $B \in F_i$.
\end{enumerate} 

       If $\mathfrak b$ is defined, then:
\begin{enumerate}
\item  $A^Q = h_i \cup \{ (\kappa, {\mathfrak b}) \}$.
\item  $B^Q$ is some upper part such that $(\kappa + 1) \cap \dom(B^Q) = 0$ and
    $B^Q \le j(B)$ for all upper parts $B \in F_i$.
\end{enumerate}

\end{itemize}

   In the case when the Boolean value $\mathfrak b$ is not defined, it is routine to check that
   $Q$ is a condition in $j({\mathbb Q}_i)$ and $Q \le j`` g_i$. This is essentially the argument of Lemma \ref{compactlemma}
   applied to the directed (hence linked) set $j`` g_i$. In the case that $\mathfrak b$ is 
   defined, the definition of $\mathfrak b$ ensures that we will still have $Q \le j`` g_i$
   so long as we can verify that $Q$ is a condition. As usual the only issue is Clause
   \ref{crucialclause}) in the definition of conditionhood. 

   So suppose that $\eta \in \dom(h_i) \cap \kappa$, 
\[
    f^Q_{\eta, j(\beta)}(j(\zeta)) = f^Q_{\eta, j(\beta')}(j(\zeta)) \neq f^Q_{\eta, j(\beta)}(j(\zeta')) = f^Q_{\eta, j(\beta')}(j(\zeta')),
\]
   and $y$ is a lower part which is harmonious with $A^Q$ past $\eta$. Let $x$ be the largest initial segment
   of $y$ which lies in $V_\kappa$, so that either $y = x$ or 
   $y = x^\frown \langle \kappa, p \rangle$ where $p \le {\mathfrak b}$. 

   By elementarity and the definition of $Q$,
\[
    f^i_{\eta, \beta}(\zeta) = f^i_{\eta, \beta'}(\zeta) \neq f^i_{\eta, \beta}(\zeta') = f^i_{\eta, \beta'}(\zeta').   
\]
   We now choose $q \in g_i$ such that $\eta < \rho^q$, $\zeta, \zeta' \in a^q$, $\beta, \beta' \in b^q$,
   and $\dom(A^q)$ contains every ordinal in $[\eta, \kappa)$ which is mentioned in $y$.
   It is easy to see that $x$ is harmonious with $A^q$ past $\eta$, and
   hence (as $q$ is a condition)
 $(x, B^q) \forces
   \zeta' {\dot G}_{\beta} \zeta \iff 
   \zeta' {\dot G}_{\beta'}  \zeta$.

   By elementarity
 $(x, j(B^q)) \forces
   j(\zeta')  j({\dot G}_{\beta}) j(\zeta) \iff 
   j(\zeta') j({\dot G}_{\beta'})  j(\zeta)$.
   To finish we just observe that by definition (and the choice of $\mathfrak b$
   in the case when it is defined, which ensures that
   $\mathfrak b \le j(B^q)(\kappa)$ ) $(y, B^Q) \le (x, j(B^q))$.

   Having chosen $Q$ as above, we let $r_{i+1}= r'$, $a_{i+1} = a'$, and 
   $q_{i+1}$ be the unique condition such that $q_{i+1} \restriction j(i) = q'$
   and $q_{i+1}(j(i)) = Q$.    
  
% CRAP NOTN CHANGE r, NOW I AM NOT SURE WHY. OH TOO MANY RS
%  ALSO $\mathbb Q$ should be ${\mathbb Q}^*$, the generic $G$ should be $G^*$ and
%  its lift should be $H^*$. 
%  ALSO TOO MANY H's. 

    At the end of the construction, we obtain  $(r^*, a^*, q^*) \in j({\mathbb L})/(G_0 * G_1 * G^*)  * j({\mathbb A})  * j({\mathbb Q}^*)$
%   CHECK NOTN
  such that 
\[
    r^* \forces a^* \le j`` G_1,
\]
and
\[
    (r_i, a_i) \forces q^* \le j`` G^*.
\]
    Forcing below $(r^*, a^*, q^*)$ we obtain a generic object $H_{\rm tail} * H_1 * H^*$ and a lifted embedding
    $j: V[G] \rightarrow M[j(G_0 * G_1) * H^*]$.    
  Following the idea of the Laver construction we
    define $U = \{ A \in (P_\kappa \kappa^{+4})^{V[G]} : j`` \kappa^{+4} \in j(A) \}$.
  Since  $H_{\rm tail} * H_1 * H^*$ is generic over $V[G]$ for highly closed forcing we have 
   $U \in V[G]$, and so $U$ is an ultrafilter witnessing the $\kappa^{+4}$ supercompactness
    of $\kappa$ in $V[G]$. By the results in Section \ref{constraintsection}, we may use $U$ to define
    a $U$-constraint $H$ such that $\Fil(H)$ is an ultrafilter.  It is easy to check that
    if $U_0$ is the projection of $U$ to a normal measure on $\kappa$,
    and  $F$ is the set of upper parts associated with $\Fil(H)$, then
\[
     F = \{ h : \mbox{$h$ is a $U_0$-constraint and $j(h)(\kappa) \ge j(H)(j``\kappa^{+4})$} \}.
\] 

    By the diamond property, there is a stationary set of $i \in \kappa^{+4} \cap \cof(\kappa^{+2})$ such that
    $U_0 \cap V[G \restriction i] = W_i$ and $F \cap V[G \restriction i] = F_i$. 
    For each such $i$, we observe that for all $h$ 
\[
    h \in F_i \implies h \in F \implies j(h)(\kappa) \ge j(H)(j``\kappa^{+4}).
\]
    So $\{ j(h)(\kappa) : h \in F_i \}$ has a nonzero lower bound, and 
    by the Truth Lemma there is a condition in  $H_{\rm tail} * H_1 * H^*_i$  which extends $(r_i, a_i, q_i)$ and forces this.
    So when we chose    $q_{i+1}$,
 we arranged that $j(h_i)(\kappa)$ is the Boolean infimum of $\{ j(h)(\kappa) : h \in F_i \}$.
    Since  $j(H)(j``\kappa^{+4})$ is a lower bound for this set, 
   $j(h_i)(\kappa) \ge j(H)(j``\kappa^{+4})$, and hence $h_i \in  F$ as required.

\end{proof} 

\section{Universal graphs} \label{graphsection}

    We are now ready to prove the main result. By the results of Section \ref{mainsection}, we will
    assume that we have in $V[G]$ a measure $U_0$ on $\kappa$, a  filter $\mathcal F$ (with associated 
    set of upper parts $F$)  and a stationary set $S$ such that for every $i \in S$:
\begin{enumerate}
\item $U_0 \cap V[G \restriction i] = W_i$.
\item $F \cap V[G \restriction i] = F_i$. 
\item $h_i \in F$.
\end{enumerate} 

    We now let $\mathbb P$ be the forcing poset defined from $U_0$ and $\mathcal F$ as in Section \ref{pforcing},
    and force with $\mathbb P$ over $V[G]$, obtaining a generic sequence
\[
    x =  f_0, \kappa_1, f_1, \kappa_2, f_2 \ldots 
\]
  By the characterisation of genericity from Lemma \ref{genericitylemma}, we see that for every $i \in S$
   $x$ is ${\mathbb P}_i$-generic over $V[G \restriction i]$, where ${\mathbb P}_i$
   is the forcing defined in $V[G \restriction i]$ from $W_i$ and $F_i$. 

   We now define for each $i \in S$ a graph ${\mathcal U}_i \in V[G][x]$, which will embed every
   graph on $\kappa^+$ in $V[G \restriction i][x]$. We begin by using the criterion for genericity
   to choose some $j$ such that $\kappa_k \in \dom(h_i)$ and $h_i(\kappa_k) \in f_k$
   for all $k \ge j$. We set $\eta = \kappa_j$.

%  NOTE TP SELF. BE CLEARER ABOUT ROLE OF CLAUSE FIVE.
%   IT I NECESSARY FOR MC AT LEAST

 The underlying set of the graph
   ${\mathcal U}_i$ is $T \times \kappa$, and the edges are defined as follows:

\medskip

\noindent   $(z, \delta)  {\mathcal U}_i (z', \delta')$ 
  if and only if there exist a lower part $t$ and 
    a condition $q \in {\mathbb Q}_i$ such that
\begin{enumerate}
\item  $q \in g_i$.
\item  $(t, B^q)$ is in the generic filter on ${\mathbb P}_i$ corresponding to the generic sequence $x$.
\item  $t$ is harmonious with $A^q$ past $\eta$. 
\item  There exist $\beta \in b^q$ and distinct $\zeta, \zeta' \in a^q$ such that:
\begin{enumerate}
\item  $f^q_{\eta, \beta}(\zeta) =  (z, \delta)$, $f^q_{\eta, \beta}(\zeta') =  (z', \delta')$,
 and $(t, B^q) \forces \zeta   {\dot {\mathcal G}}^i_\beta \zeta'$. 
  \end{enumerate} 
\end{enumerate}

   Since $\langle  {\dot {\mathcal G}}^i_\beta : \beta < \kappa^{+3} \rangle$ enumerates all ${\mathbb P}_i$-names
   for graphs on $\kappa^+$, it will suffice to verify that the generic function $f^i_{\eta, \beta}$ 
   is an embedding of ${\mathcal G}^i_\beta$ (the realisation of the name ${\dot {\mathcal G}}^i_\beta$) into the graph ${\mathcal U}_i$.
   One direction is easy:  
   if $f^i_{\eta, \beta}(\zeta)  {\mathcal U}_i f^i_{\eta, \beta}(\zeta')$
   then by definition there is $(t, B^q)$ in the generic filter on ${\mathbb P}_i$ induced by $x$ such that
   $(t, B^q) \forces \zeta   {\dot {\mathcal G}}^i_\beta \zeta'$,    
   and so by the Truth Lemma $\zeta  {\mathcal G}^i_\beta \zeta'$.

   For the converse direction, suppose that $\zeta  {\mathcal G}^i_\beta \zeta'$.
   We may find a condition $(s, B)$ in the generic filter induced by $x$ on ${\mathbb P}_i$,
   such that $(s, B) \forces \zeta   {\dot {\mathcal G}}^i_\beta \zeta'$.
   Let $s =  q_0, \kappa_1, q_1, \ldots, \kappa_n, q_n$. Extending the condition $(s, B)$ if need be, we may assume that
   $n \ge j$.  Since $(s, B)$ is in the generic filter induced by $x$, we have that
   $q_m \in f_m$ for $m \le n$, while $\kappa_m \in \dom(B)$ and $B(\kappa_m) \in f_m$ for $m > n$.  

   By the properties of the forcing poset ${\mathbb Q}_i$, we may find a condition
   $q$ in $g_i$ such that $\ssup(\dom(A^q)) > \kappa_n$, $\dom(B^q) \subseteq \dom(B)$,
   $B^q(\alpha) \le B(\alpha)$ for all $\alpha \in \dom(B)$, $\beta \in b^q$ and
   $\zeta, \zeta' \in a^q$.  

%
%   Stuff I need:  $t$ is harm with $A^q$ past eta, $(t, B^q)$ extends $(s, B)$, $(t, B^q)$ in x filter 
%  

   Recall now that $\kappa_k \in \dom(h_i)$ and $h_i(\kappa_k) \in f_k$
   for all $k \ge j$, and also that $\eta = \kappa_j$ and $n \ge j$.
   Let $\bar n$ be the largest $k$ such that $\kappa_k < \ssup(\dom(A^q))$.

   Define a lower part $t$ as follows: 
\[
   t = q'_0, \kappa_1, q'_1, \ldots, \kappa_{\bar n}, q'_{\bar n}
\]
   where:
\begin{enumerate}
\item $q'_k = q_k$ for $k < j$.
\item $q'_k = q_k \cup A^q(\kappa_k)$ for $j \le k \le n$.
\item $q'_k = A^q(\kappa_k) \cup B(\kappa_k)$ for $n < k \le {\bar n}$.
\end{enumerate}

    We note that since $q \in g_i$ and $h_i$ is added by $g_i$,
    $A^q$ is an initial segment of $h_i$ and 
    $h_i(\alpha) \le B^q(\alpha)$ for all $\alpha \in \dom(h_i) \setminus \dom(A^q)$. 
   For $j \le k \le n$ we have that $q_k \in f_k$ and
    $A^q(\kappa_k) = h_i(\kappa_k) \in f_k$, so that 
    $q_k \cup A^q(\kappa_k)$ is a condition and lies in $f_k$.
    For $n < k \le {\bar n}$, again $A^q(\kappa_k) = h_i(\kappa_k) \in f_k$ and also
    $B(\kappa_k) \in f_k$, so that $A^q(\kappa_k) \cup B(\kappa_k)$ is a condition and lies  in $f_k$.

  We will verify that $t$ is harmonious with $A^q$ past $\eta$,
 $(t, B^q)$ extends $(s, B)$, and $(t, B^q)$ is in the filter generated by $x$. 
 This will suffice, since it will then be clear that $t$ and $q$ will serve as witnesses that
 $f^i_{\eta, \beta}(\zeta)  {\mathcal U}_i f^i_{\eta, \beta}(\zeta')$.

  The harmoniousness is immediate from the definitions. $(t, B^q)$ extends $(s, B)$ because
  $q'_k \le q_k$ for $k \le n$, $q'_k \le B(\kappa_k)$ for $n < k \le {\bar n}$,
  and $B^q \le B$. We already checked that $q'_k \in f_k$ for all $k \le {\bar n}$,
  so to finish we just need to see that $B^q(\kappa_k) \in f_k$ for all $k > {\bar n}$;
  this is immediate because $h_i(\kappa_k) \le B^q(\kappa_k)$ for all such $k$, and
  also $h_i(\kappa_k)\in f_k$ for all $k \ge j$.    

%MAKE IT A LEMA 

  To finish the construction of a small family of universal graphs, we will fix
  $i^* \in S$ which is a limit of points of $S$, and an increasing
  $\kappa^{++}$-sequence of points $i_\eta \in S$ which is cofinal in $i^*$.
  By routine chain condition arguments, every ${\mathbb P}_{i^*}$-name for 
  a graph on $\kappa^+$ may be viewed as a 
  ${\mathbb P}_{i_\eta}$-name for some $\eta < \kappa^{++}$. 
  We now consider the model $V[G \restriction i^*][x]$. The 
  family of graphs $\{  {\mathcal U}_{i_\eta} : \eta < \kappa^{++} \}$ is universal
  in this model, where $2^\kappa = 2^{\kappa^+} = \kappa^{+3}$ and
  of course $\kappa = \aleph_\omega$. 

%  EXPLAIN A BIT CHAIN CDN CATCH YOUR TAIL.
%  ALSO REFL OF PROPS I to III, want to GUESS AT REFL POINT.
 
 We have proved:

\begin{theorem} \label{mainthm} It is consistent from large cardinals that 
 $\aleph_\omega$ is strong limit, $2^{\aleph_\omega} = 2^{\aleph_{\omega+1}} =  \aleph_{\omega+3}$,
 and there is a family of size $\aleph_{\omega+2}$ of graphs on
 $\aleph_{\omega+1}$ which is jointly universal for all such graphs.
\end{theorem}

\section{Afterword} \label{after}

   There is some flexibility in the proof of Theorem \ref{mainthm}, in particular
   it would be straightforward to modify the construction so that in the final
   model $2^{\aleph_\omega} = \aleph_{\omega +k}$ for an arbitrary $k$ such that $3 \le k < \omega$.
   Larger values can probably be achieved but would require a substantial modification to the
   construction.  

   Theorem \ref{mainthm} leaves a number of natural questions open:
\begin{itemize}
\item Can we have a failure of SCH at $\aleph_\omega$ with $u_{\aleph_{\omega+1}} = 1$? 
\item On a related topic, what is the exact value of $u_{\aleph_{\omega+1}}$ in the model of Theorem
 \ref{mainthm}?
\item As far as the authors are aware, the only known results on the value of
   $u_{\kappa^+}$ for $\kappa$ singular strong limit and $2^\kappa > \kappa^+$ are consistency results
   of the kind proved in this paper. In particular, we lack a forcing technique to show that
   $u_{\kappa^+}$ can be arbitarily large.

  For $\kappa$ regular adding Cohen subsets to $\kappa$ makes $u_{\kappa^+}$
  arbitrarily large, is there an analogous result for $\kappa$ singular?
\item The class of graphs is a very simple class of structures. What can be done in
  more complex classes?
\item  In the model of Theorem \ref{mainthm}, GCH fails cofinally often below $\aleph_\omega$, and in
   fact $2^{\aleph_n} = \aleph_{n+4}$ for unboundedly many $n < \omega$. Is the conclusion consistent
   if we demand that GCH holds below $\aleph_\omega$?  
\end{itemize} 

   The authors' joint paper with Magidor and Shelah \cite{5authors} contains some related work,
   in which the final ``Prikry type'' forcing is a version of Radin forcing and we obtain
   models where $\mu$ is singular strong limit  of uncountable cofinality, SCH fails at $\mu$ and
   $u_{\mu^+} < 2^\mu$.

\bibliography{3authorsbib}

\begin{thebibliography}{10}

\bibitem{James}
James Cummings.
\newblock A model in which {GCH} holds at successors but fails at limits.
\newblock {\em Transactions of the American Mathematical Society},
  329(1):1--39, 1992.

\bibitem{5authors}
James Cummings, Mirna D{\v z}amonja, Menachem Magidor, Charles Morgan, and
  Saharon Shelah.
\newblock A framework for forcing constructions at successors of singular
  cardinals.
\newblock Submitted.

\bibitem{DzSh2}
Mirna D{\v z}amonja and Saharon Shelah.
\newblock Universal graphs at the successor of a singular cardinal.
\newblock {\em Journal of Symbolic Logic}, 68:366--387, 2003.

\bibitem{DzSh}
Mirna D{\v{z}}amonja and Saharon Shelah.
\newblock On the existence of universal models.
\newblock {\em Archive for Mathematical Logic}, 43(7):901--936, 2004.

\bibitem{FoWo}
Matthew Foreman and Hugh Woodin.
\newblock The generalized continuum hypothesis can fail everywhere.
\newblock {\em Annals of Mathematics}, 133(1):1--35, 1991.

\bibitem{KoSh}
Menachem Kojman and Saharon Shelah.
\newblock Nonexistence of universal orders in many cardinals.
\newblock {\em The Journal of Symbolic Logic}, 57(3):875--891, 1992.

\bibitem{Laver}
Richard Laver.
\newblock Making the supercompactness of $\kappa$ indestructible under
  $\kappa$-directed closed forcing.
\newblock {\em Israel Journal of Mathematics}, 29(4):385--388, 1978.

\bibitem{SCH1}
Menachem Magidor.
\newblock On the singular cardinals problem. {I}.
\newblock {\em Israel Journal of Mathematics}, 28:1--31, 1977.

\bibitem{Mathias}
Adrian Mathias.
\newblock Sequences generic in the sense of {P}rikry.
\newblock {\em Journal of the Australian Mathematical Society}, 15(4):409--414,
  1973.

\bibitem{Mekler}
Alan Mekler.
\newblock Universal structures in power $\aleph_1$.
\newblock {\em Journal of Symbolic Logic}, 55(2):466--477, 1990.

\bibitem{Mitchell}
William~J. Mitchell.
\newblock How weak is a closed unbounded ultrafilter?
\newblock In {\em Logic Colloquium '80 (Prague, 1980)}, volume 108 of {\em
  Studies in Logic and the Foundations of Mathematics}, pages 209--230.
  North-Holland, Amsterdam, 1982.

\bibitem{ShFA}
Saharon Shelah.
\newblock A weak generalization of {MA} to higher cardinals.
\newblock {\em Israel Journal of Mathematics}, 30(4):297--306, 1978.

\bibitem{Shgraphs}
Saharon Shelah.
\newblock On universal graphs without instances of {CH}.
\newblock {\em Annals of Pure and Applied Logic}, 26(1):75--87, 1984.

\bibitem{Shelahdiamonds}
Saharon Shelah.
\newblock Diamonds.
\newblock {\em Proceedings of the American Mathematical Society},
  138:2151--2161, 2010.

\end{thebibliography}
\bibliographystyle{plain}

\end{document}